\documentclass{article}

\usepackage{amsmath,amsthm,amssymb}
\usepackage{a4wide}
\usepackage[utf8]{inputenc}
\usepackage{hyperref}

\usepackage{enumitem}
\usepackage{color}
\usepackage{graphicx}

\SetLabelAlign{CenterWithParen}{\hfil(\makebox[1.0em]{#1})\hfil}

\def\eps{\varepsilon}

\def\rr{\mathbb R}

\def\nn{\mathbb N}

\DeclareMathOperator{\sgn}{sgn}

\providecommand{\abs}[1]{\lvert#1\rvert}
\providecommand{\norm}[1]{\lVert#1\rVert}

\newtheorem{theorem}{Theorem}[section]
\newtheorem{lemma}[theorem]{Lemma}
\newtheorem{proposition}[theorem]{Proposition}

\theoremstyle{definition}
\newtheorem{defi}{Definition}

\theoremstyle{remark}
\newtheorem{remark}{Remark}

\numberwithin{equation}{section}
\begin{document}

\title{\textbf{Error estimates for finite difference schemes
    associated with Hamilton-Jacobi equations on a junction}}
\author{Jessica Guerand\footnote{ D\'epartement de Mathématiques et applications, \'Ecole Normale Sup\'erieure, CNRS, PSL Research University, 
45 rue d’Ulm, 75005 Paris, France. \texttt{jessica.guerand@ens.fr}} \mbox{} and Marwa
Koumaiha\footnote{LAMA, UMR 8050, Univ. Paris-Est Créteil, 61 avenue
  du général de Gaulle, 94010 Créteil cedex, France \& Université
  Libanaise, École Doctorale des Sciences et de Technologie, Hadath,
  Beirut, Liban, \texttt{marwa.koumaiha@math.cnrs.fr}}}

\date{\today}

\maketitle

\begin{abstract}
  This paper is concerned with monotone (time-explicit) finite
  difference schemes associated with first order Hamilton-Jacobi
  equations posed on a junction. They extend the schemes recently
  introduced by Costeseque, Lebacque and Monneau (2013) to general
  junction conditions. On the one hand, we prove the convergence of the
  numerical solution towards the viscosity solution of the
  Hamilton-Jacobi equation as the mesh size tends to zero for general
  junction conditions. On the other hand, we derive optimal error estimates of
  order $(\Delta x)^{\frac12}$ in $L^\infty_{\text{loc}}$ for junction
  conditions of optimal-control type at least if the flux is ``strictly limited''. 
\end{abstract}



\setcounter{tocdepth}{1}
\tableofcontents

\section{Introduction}

This paper is concerned with numerical approximation of first order
Hamilton-Jacobi equations posed on a junction, that is to say a
network made of one node and a finite number of edges. 

The theory of viscosity solutions for such equations on such
domains has reached maturity by now \cite{schieborn,CS,ACCT,IMZ,IM}.
In particular, it is now understood that general junction conditions
reduce to special ones of optimal-control type \cite{IM}. Roughly
speaking, it is proved in \cite{IM} that imposing a junction condition
ensuring the existence of a continuous viscosity solution and a
comparison principle is equivalent to imposing a junction condition
obtained by ``limiting the flux'' at the junction point.

For the ``minimal''\emph{flux-limited} junction conditions,
Costeseque, Lebacque and Monneau \cite{CLM} introduced a monotone
numerical scheme and proved its convergence. Their scheme can be
naturally extended to general junction conditions and our first
contribution is to introduce it and to prove its convergence.

Our second and main result is an error estimate à la Crandall-Lions
\cite{CL} in the case of flux-limited junction conditions. It is
explained in \cite{CL} that the proof of the comparison principle
between sub- and super-solutions of the continuous Hamilton-Jacobi
equation can be adapted in order to derive error estimates between the
numerical solution associated with monotone (stable and consistent)
schemes and the continuous solution.  In the Euclidian case, the
comparison principle is proved thanks to the technique of doubling
variables; it relies on the classical penalisation term $\eps^{-1}
|x-y|^2$. Such a penalisation procedure is known to fail in general if
the equation is posed in a junction; it is explained in \cite{IM}
that it has to be replaced with a \emph{vertex test function}. Here we replace it by the \emph{reduced minimal action} introduced in \cite{IMZ} for the ``minimal'' \emph{flux-limited} junction conditions. We study and use it in the case where the flux is ``strictly'' limited. 

In order to derive error estimates as in \cite{CL}, it is
important to study the regularity of the test function. More
precisely, we prove (Proposition~\ref{prop:c11}) that its gradient is locally Lipschitz
continuous, at least if the flux is ``strictly limited'' and far away from a special curve. But we also see that the reduced minimal action is not of class $C^{1}$ on this curve. However we can get ``weaker'' viscosity inequalities thanks to a result in \cite{IM} (see Proposition \ref{propimp}). Such a
regularity result is of independent interest.

\subsection{Hamilton-Jacobi equations posed on junctions} 

A \textit{junction} is a network made of one node and a finite number
of infinite edges. It can be viewed as the set of $N$ distinct copies
$(N \ge 1)$ of the half-line which are glued at the origin.  For
$\alpha=1,\dots,N,$ each branch $J_{\alpha}$ is assumed to be isometric
to $[0,+\infty)$ and
\[J=\bigcup_{\alpha=1,\ldots,N} J_\alpha \quad \text{ with } \quad J_{\alpha}\cap J_{\beta} =\{0\} 
\quad \text{for} \quad \alpha \neq \beta\]
where the origin $0$ is called the \textit {junction point}. 
For points $x,y\in J$, $d(x,y)$ denotes the geodesic distance on $J$ defined as
\[d(x,y)=\left\lbrace
\begin{array}{ll}
\abs{x-y} & \text{ if } x,y  \text{ belong to the same branch},\\
\abs{x}+\abs{y} &\text{ if } x,y \text{ belong to different branches.}
\end{array}\label{zizou}
\right.\] 
With such a notation in hand, we consider the following
Hamilton-Jacobi equation posed on the junction $J$,
\begin{equation}\label{eq:main}
\left\lbrace
\begin{array}{ll}
  u_t+H_{\alpha}(u_x)=0 & \text{ in } (0,T)\times J_\alpha \setminus \{0\}, \\
  u_t +F(\frac{\partial u}{\partial x_{1}},\dots,\frac{\partial u}{\partial x_{N}})=0 
  & \text{ in } (0,T)\times \{0\}, 
\end{array}
\right.
\end{equation}
submitted to the initial condition
\begin{equation}
\label{eq:ic}
u(0,x)=u_0(x), \quad \text{ for }  x\in J
\end{equation} 
where $u_0$ is globally Lipschitz in $J$. The second equation in
\eqref{eq:main} is referred to as the \textit{junction condition}. We
consider the important case of quasi-convex Hamiltonians $H_{\alpha}$ satisfying
the following conditions:
\begin{equation}
\label{H}
\text{There exists $p_0^\alpha \in \rr$ such that }
\left\lbrace \begin{array}{ll}
 H_{\alpha}\in C^2 (\rr) \text{ and } H''_\alpha (p_0^\alpha) >0   \\
 \pm H'_\alpha (p) > 0 \text{ for } \pm (p-p_0^\alpha) > 0 \\
 \lim_{\abs{p}\to +\infty}H_{\alpha}(p)=+\infty .
\end{array}
\right.
\end{equation}
In particular
$H_{\alpha}$ is non-increasing in $(-\infty,p_0^{\alpha}]$ and
non-decreasing in $[p_0^{\alpha},+\infty)$, and we set
\[H_{\alpha}^-(p)=
\begin{cases}
H_{\alpha}(p) & \text{ for } p \le p_0^{\alpha}\\ 
H_{\alpha}(p_0^{\alpha}) & \text{ for }  p\ge p_0^{\alpha}
\end{cases}
\quad \text{ and } \quad
H_{\alpha}^+(p)=
\begin{cases}
H_{\alpha}(p_0^{\alpha}) & \text{ for } p \le p_0^{\alpha},\\ 
H_{\alpha}(p) & \text{ for } p\ge p_0^{\alpha}
\end{cases}\]
where $H_{\alpha}^-$ is non-increasing and $H_{\alpha}^+$ is non-decreasing.

We next introduce a one-parameter family of junction conditions: given
a \emph{flux limiter} $A\in\rr\cup\{-\infty\},$ the A-limited flux
junction function  is defined for $p=(p_1,\dots,p_N)$ as,
\begin{equation}
\label{fa}
F_{A}(p)=\max\bigg(A,\max_{\substack{\alpha =1,\dots,N}}
H_{\alpha}^{-}(p_{\alpha})\bigg)
\end{equation}
for some given $A \in \rr \bigcup \{-\infty\}$ where $H_{\alpha}^{-}$
is non-increasing part of $H_{\alpha}$.

We now consider the following important special case of \eqref{eq:main},
\begin{equation}
\label{eq:continuous}
\left\lbrace
\begin{array}{ll}
  u_t+H_{\alpha}(u_x)=0 & \text{ in } (0,T)\times J_\alpha \setminus \{0\}, \\
  u_t +F_A(\frac{\partial u}{\partial x_{1}},\dots,\frac{\partial u}{\partial x_{N}})=0 & \text{ in } (0,T)\times \{0\}.
\end{array}
\right.
\end{equation}
We point out that all the junction functions $F_A$ associated with $A\in [-\infty,A_0]$ coincide if one chooses 
\begin{equation}
\label{a0}
A_0= \max_{\substack{\alpha=1,\dots,N}}\min_{\substack{\rr}} H_{\alpha}.
\end{equation}
As far as general junction conditions are concerned, we assume that
the junction function $F:\rr^{n}\mapsto \rr$ satisfies
 \begin{equation}
\label{F}
\left\lbrace\begin{array}{l} F\text{ is continuous and piecewise } C^{1}(\rr^n),\\
                \forall \alpha, \forall p=(p_1,\dots,p_N) \in \mathbb{R}^{N}, \frac{\partial F}{\partial p_{\alpha}}(p)<0, \\
               F (p_1,\dots,p_N) \to + \infty \text{ as } \displaystyle \min_{i\in \{1,\dots,N\}} p_i \to -\infty.
 \end{array}\right.
\end{equation}   

\paragraph{\textbf{Hypothesis in the following of the paper: $p_{0}^{\alpha}=0$.}}
Without loss of generality (see \cite[Lemma 3.1]{IM}), we consider in this paper that $p_{0}^{\alpha}=0$ for $\alpha=1,...,N$, i.e.,  
\begin{equation}
\label{minH}
\min H_{\alpha}=H_{\alpha}(0).
\end{equation}
 Indeed, $u$ solves \eqref{eq:continuous} if and only if
$\tilde{u}(t, x) := u(t, x)-p_{0}^{\alpha}x$ for $x\in J_{\alpha}$ solves the same equation in which $H_{\alpha}$ is replaced by $\tilde{H}_{\alpha}(p)=H_{\alpha}(p+p_{0}^{\alpha})$. We have the same result for $u^{h}$ the solution of the scheme \eqref{scheme}.

\paragraph{The optimal control framework.} It is well known that the Legendre-Fenchel conjugate is crucial in establishing a link between the general Cauchy problem \eqref{eq:continuous}-\eqref{eq:ic} and a control problem \cite{Lions}. Through this link, we obtain the representation formula for the exact solution.
Before treating the case where the Hamiltonians $H_{\alpha}$ satisfy \eqref{H}, we first consider the case of Hamiltonians satisfying the hypotheses of \cite{IMZ} i.e.,
\begin{equation}
\label{Jconvex1}
\left\lbrace \begin{array}{ll}
\text{(\textbf{Regularity})} & H_{\alpha}\text{ is of class 
} C^2   \\
\text{(\textbf{Coercivity})} & \lim_{\abs{p}\to +\infty}H_{\alpha}(p)=+\infty \\
\text{(\textbf{Convexity})} &H_\alpha \text{ is convex  and is the Legendre Fenchel transform of } L_{\alpha} \\
 &\text{ where } L_{\alpha} \text{ is of class } C^{2} \text{ and satisfies (B0)}.
\end{array}
\right.
\end{equation}
 
We recall that 
\begin{equation}
\label{J12}
H_{\alpha}(p)=L_{\alpha}^{\star}(p)=\sup_{q\in \mathbb{R}} (pq-L_{\alpha}(q)).
\end{equation}
We consider the following hypothesis for $L_{\alpha}$,
\begin{center}
\textbf{(B0)} There exists a constant $\gamma>0$ such that for all $\alpha=1,\cdots,N,$ the $C^2(\mathbb{R})$ functions $L_{\alpha}$ satisfy $L_{\alpha}^{\prime \prime}\ge \gamma>0.$ 
\end{center}
 
An optimal control interpretation of the Hamilton-Jacobi equation \eqref{eq:continuous} is given in \cite{B,bar-dol,Lions,collegelions}. We define the set of admissible controls at a point $x\in J$ by
\[
\mathcal{U}(x)=\left\lbrace
\begin{array}{ll}
\mathbb{R}e_{\alpha_{0}} &\text{ if } x\in J_{\alpha_{0}}^{\star},\\
\cup_{\alpha=1,\cdots N}\mathbb{R}^{+} e_{\alpha} &\text{ if } x=0.
\end{array}
\right.
\] 
For $(s,y), (t,x)\in [0,T]\times J$ with $s\le t,$ we define the set of admissible trajectories from $(s,y)$ to $(t,x)$ by
\begin{equation}
\label{J18}
\mathcal{A}(s ,y;t,x)=\left\{
X\in W^{1,1}([s,t], \mathbb{R}^2): \quad \vrule 
\begin{array}{ll}
X(\tau)\in J & \text{ for all } \tau\in (s,t)\\
\dot{X} (\tau) \in \mathcal{U}(X(\tau)) & \text{ for a.e } \tau\in (s,t)\\
X(s)=y & \text{ and } X(t)=x
\end{array}
\right\}.
\end{equation}
For $P=pe_i\in \mathcal{U}(x)$ with $p\in \mathbb{R},$ we define the Lagrangian on the junction
\begin{equation}
\label{J17}
L(x,p)=\left\lbrace\begin{array}{ll}
L_{\alpha}(p) & \text{ if } x\in J_{\alpha}^{\star},\\
L_{A}(p) & \text{ if } x=0,
\end{array}
\right.
\end{equation}

with
\begin{equation*}
L_A(p)= \min\bigg(-A,\min_{\substack{\alpha =1,\dots,N}}
L_{\alpha}(p)\bigg).
\end{equation*}

The Hopf-Lax representation formula of the solution of \eqref{eq:continuous}-\eqref{eq:ic} is given in \cite{IMZ,ad-mi-go} by

\begin{equation}
\label{J19}
u_{oc}(t,x)= \inf_{\substack{{y\in J}}}\{u_0(y)+\mathcal{D}(0,y;t,x)\} 
\end{equation}
with 
\[
\mathcal{D}(0,y;t,x)=\inf_{\substack{{X\in \mathcal{A}(0,y;t,x)}}}\bigg\{\int_0^t L(X(\tau),\dot{X}(\tau))d\tau\bigg\}.
\]

\subsection{Presentation of the scheme}

The domain $(0,+\infty) \times J$ is discretized with respect to time
and space. We choose a regular grid in order to simplify the
presentation but it is clear that more general meshes could be used
here. The space step is denoted by $\Delta x$ and the time step by
$\Delta t$. If $h$ denotes $(\Delta t,\Delta x)$, the mesh (or
grid) $\mathcal{G}_h$ is chosen as
\[\mathcal{G}_h = \{n \Delta t : n \in \nn \} \times J^{\Delta x}\] 
where 
\[ J^{\Delta x} = \bigcup_{\alpha=1,\dots,N} J_\alpha^{\Delta x} \quad \text{ with }
\quad J_\alpha \supset J_\alpha^{\Delta x} \simeq \{ i \Delta x : i \in \nn \}.\]
It is convenient to write $x_i^\alpha$ for $i\Delta x \in J_\alpha$. 

A numerical approximation $u^h$ of the solution $u$ of the
Hamilton-Jacobi equation is defined in $\mathcal{G}_h$; the quantity
$u^h (n \Delta t, x_i^\alpha)$ is simply denoted by $U_i^{\alpha,
  n}$. We want it to be an approximation of $u(n\Delta t, x_i^\alpha)$
for $n \in \nn$, $i\in \nn$, where $\alpha$ stands for the index of
the  branch.

We consider the following time-explicit scheme:  for $n\ge 0$,
\begin{equation}
\label{scheme}
\left\lbrace
\begin{array}{lll}
\frac{U^{\alpha,n+1}_i-U^{\alpha,n}_i}{\Delta t}
+\max\{H_{\alpha}^{+}(p_{i,-}^{\alpha,n}),H_{\alpha}^{-}(p_{i,+}^{\alpha,n})\}=0, & i\ge 1, &
\alpha=1,\dots,N\\ 
U_0^{\beta,n}:=U_0^n, &  i=0, & \beta=1,\dots,N\\
\frac{U_0^{n+1}-U_0^{n}}{\Delta t}+F(p_{0,+}^{1,n},\dots,p_{0,+}^{N,n})=0, &  &
\end{array}
\right.
\end{equation}
where $p_{i,\pm}^{\alpha,n}$ are the discrete (space) gradients defined by
\begin{equation}
\label{dsg}
p_{i,+}^{\alpha,n}:= \frac{U_{i+1}^{\alpha,n}-{U_{i}^{\alpha,n}}}{\Delta x}, \qquad
p_{i,-}^{\alpha,n}:= \frac{U_{i}^{\alpha,n}-{U_{i-1}^{\alpha,n}}}{\Delta x}
\end{equation}
with the initial condition
\begin{equation}
\label{eq:cid}
U_i^{\alpha,0}=u_0(x_i^\alpha), \quad i\ge 0,\quad \alpha=1,\dots,N.
\end{equation}

The following Courant-Friedrichs-Lewy (CFL) condition ensures that the explicit scheme is monotone,
\begin{equation}
\label{cfl}
\frac{\Delta x}{\Delta t}\ge \max\bigg\{  
\max_{\substack{{i\ge0},\\{\alpha=1,\dots,N,}\\{0\le n\le n_T}}}{\abs{H_{\alpha}^{'}(p_{i,+}^{\alpha,n})}}; 
 \max_{\substack{0\le n\le n_T}}\bigg\{(-\nabla \cdot F)(p_{0,+}^{1,n},\dots,p_{0,+}^{N,n})\bigg\}\bigg\}
\end{equation}
where the integer $n_T$ is the integer part of
$\frac{T}{\Delta t}$ for a given $T>0$.

\subsection{Main results}

As previously noticed in \cite{CLM} in the special case $F=F_{A_0}$,
it is not clear that the time step $\Delta t$ and space
  step $\Delta x$ can be chosen in such a way that the CFL
  condition~\eqref{cfl} holds true since the discrete gradients
  $p_{i,+}^{\alpha,n}$ depend itself on $\Delta t$ and $\Delta x$
  (through the numerical scheme). We thus impose a more stringent CFL
  condition,
\begin{equation}
\label{cflr}
\frac{\Delta x}{\Delta t}\ge \max\bigg\{\max_{\substack{{\alpha=1,\dots,N},\\ 
{\underline{p}_{\alpha}\le p\le \overline{p}_\alpha}}}{\abs{H_{\alpha}'(p)}};  
\max_{\substack{{\alpha=1,\dots,N},\\ \underline{p}_{\alpha}^0\le p_{\alpha}\le \overline{p}_\alpha}}
\bigg\{(-\nabla \cdot F)(p_{1},\dots,p_{N})\bigg\}\bigg\}
\end{equation}
for some
$\underline{p}_\alpha,\overline{p}_\alpha,\underline{p}_\alpha^0 \in
\rr$
to be fixed (only depending on $u_0$, $H$, and $F$).
We can argue as in \cite{CLM} and prove that
$\underline{p}_\alpha,\overline{p}_\alpha,\underline{p}_\alpha^0 \in
\rr$
can be chosen in such a way  that the CFL condition~\eqref{cflr} implies \eqref{cfl}
and, in turn, the scheme is monotone (Lemma~\ref{lem:monotone} in
Section~\ref{sec:conv}). We will also see that it is stable
(Lemma~\ref{lem:stab}) and consistent (Lemma~\ref{lem:cons}). It is
thus known that it converges \cite{CL,BS}.
Notice that taking $F=F_{A}$, gives the following CFL condition
\begin{equation}
\label{cflr2}
\frac{\Delta x}{\Delta t}\ge \max_{\substack{{\alpha=1,\dots,N},\\
{\underline{p}_{\alpha}\le p\le \overline{p}_\alpha}}}{\abs{H_{\alpha}'(p)}}.
\end{equation}
\begin{theorem}[\textbf{Convergence  for general junction conditions}]\label{Conv-gen}
  Let $T>0$ and $u_0$ be Lipschitz continuous. There exist
  $\underline{p}_\alpha,\overline{p}_\alpha,\underline{p}_\alpha^0 \in
  \rr$, $\alpha =1,\dots, N$, depending only on the initial data, the
  Hamiltonians and the junction function $F$, such that, if $h$
  satisfies the CFL condition \eqref{cflr}, then the numerical
  solution $u^h$ defined by \eqref{scheme}-\eqref{eq:cid} converges
  locally uniformly as $h$ goes to zero to the unique relaxed viscosity solution $u$ of \eqref{eq:main}-\eqref{eq:ic},
  on any compact set $\mathcal{K} \subset [0, T )\times J$, i.e.
\begin{equation}
\label{convergence}
\limsup_{|h| \to 0} \sup_{\substack{(t,x)\in \mathcal{K} \cap \mathcal{G}_h}}\abs{u^h (t,x)-u(t,x)}=0.
\end{equation}
\end{theorem}

\begin{remark}
We know from \cite{IM} that the equation \eqref{eq:main}-\eqref{eq:ic} may have no viscosity solution but always a unique relaxed viscosity solution.
Notice that the scheme has a junction condition which is not relaxed. However the solution of the scheme converges to the unique relaxed solution of the associated Hamilton-Jacobi equation. 
\end{remark}
The main result of this paper lies in getting error estimates in the
case of flux-limited junction conditions.
\begin{theorem}[\textbf{Error estimates for flux-limited junction conditions}]
\label{thm:ee}
Let $T>0$ and $u_0$ be Lipschitz continuous, $u^h$ be the solution of the
associated numerical scheme \eqref{scheme}-\eqref{eq:cid} and $u$ be
the viscosity solution of \eqref{eq:continuous}-\eqref{eq:ic}
for some $A \in \rr$. If the CFL condition \eqref{cflr2} is satisfied,
then there exists $C>0$ (independent of $h$) such that
\begin{equation}\label{eq:ee}
 \sup_{\substack{[0,T)\times J \cap \mathcal{G}_h}}\abs{u^h(t,x)-u(t,x)}\le 
\begin{cases}
C(\Delta x)^{1/2} & \text{ if } A > A_0, \\
C (\Delta x)^{2/5} & \text{ if } A = A_0.
\end{cases}
\end{equation}
\end{theorem}


\subsection{Related results}

\paragraph{Numerical schemes for Hamilton-Jacobi equations on networks.}
The discretization of viscosity solutions of Hamilton-Jacobi
equations posed on networks has been studied in a few papers
only. Apart from \cite{CLM} mentioned above, we are only aware of two
other works. A convergent semi-Lagrangian scheme is introduced in
\cite{cfs} for equations of eikonal type.  In \cite{gzh}, an adapted
Lax-Friedrichs scheme is used to solve a traffic model; it is worth
mentioning that this discretization implies to pass from the scalar
conservation law to the associated Hamilton-Jacobi equation at each
time step.

For optimal control problems, the numerical approximation of (HJ) has already been studied using schemes
based on the discrete dynamic programming principle. Essentially, these schemes are
built by replacing the continuous optimal control problem by its discrete time version. We refer to Capuzzo Dolcetta \cite{C}, Capuzzo Dolcetta-Ishii
\cite{CDI} for the results concerning the convergence of $u_h$ to $u$ and the a priori estimates (of order $\Delta x)$ ,
in the $L^{\infty},$ giving the order of convergence of the discrete-time approximation.
We refer to Falcone \cite{F} for the results related to the order of convergence of the
fully discrete (i.e. in space and time) approximation and for the construction of the algorithm, we mention that under a semiconcavity assumption the rate of convergence is of order 1. We cite also  \cite{Fal-fer} and references therein for discrete time high order schemes for Hamilton Jacobi Bellman equations.

\paragraph{Link with monotone schemes for scalar conservation laws.} 
We first follow \cite{CLM} by emphasizing that the convergence result,
Theorem~\ref{Conv-gen}, implies the convergence of a monotone scheme
for scalar conservation laws (in the sense of distributions).

In order to introduce the scheme, it is useful to introduce a notation
for the numerical Hamiltonian $\mathcal{H}_\alpha$, 
\[ \mathcal{H}_\alpha (p^+,p^-) = \max\{H_{\alpha}^{-}(p^+),H_{\alpha}^{+}(p^-)\}. \]
The discrete solution  $(V^n)$ of the scalar conservation law is defined as
follows,
\[V ^{\alpha,n}_{i+\frac12}= 
\begin{cases} 
\frac{U^{\alpha,n}_{i+1}-U^{\alpha,n}_i}{\Delta x} & \text{ if } i \ge 1 \medskip\\
\frac{U^{\alpha,n}_1 -U^n_0}{\Delta x} & \text{ if } i = 0.
\end{cases}\]
In view of \eqref{scheme}, it satisfies for all $\alpha=1,\dots,N$, 
\begin{equation*}
\left\lbrace
\begin{array}{lll}
\frac{V^{\alpha,n+1}_{i+\frac12}-V^{\alpha,n}_{i+\frac12}}{\Delta t} 
+ (\Delta x)^{-1} \left( \mathcal{H}_\alpha (V^{\alpha,n}_{i+\frac32},V^{\alpha,n}_{i+\frac12}) - 
\mathcal{H}_\alpha (V^{\alpha,n}_{i+\frac12},V^{\alpha,n}_{i-\frac12}) \right) 
=0,  &   i\ge 1, \medskip \\ 
\frac{V_{\frac12}^{\alpha,n+1}-V_{\frac12}^{\alpha,n}}{\Delta t}
+ (\Delta x)^{-1} \left(\mathcal{H}_\alpha (V^{\alpha,n}_{\frac32},V^{\alpha,n}_{\frac12})  
- F(V_{\frac12}^{1,n},\dots,V_{\frac12}^{N,n}) \right)=0. &
\end{array}
\right.
\end{equation*}
submitted to the initial condition
\begin{equation*}
V_{i+\frac12}^{\alpha,0}=\frac{u_0(x_i^\alpha)-u_0(0)}{\Delta x}, \quad i\ge 0,\quad \alpha=1,\dots,N.
\end{equation*}
In view of Theorem~\ref{Conv-gen}, we thus can conclude that the
discrete solution $v^h$ constructed from $(V^n)$ converges towards
$u_x$ in the sense of distributions, at least far from the junction point. 

\paragraph{Scalar conservation laws with Dirichlet boundary conditions
  and constrained fluxes.}

We would like next to explain why our result can be seen as the
Hamilton-Jacobi counterpart of the error estimates obtained by
Ohlberger and Vovelle \cite{ov} for scalar conservation laws submitted
to Dirichlet boundary conditions.

On the one hand, it is known since 1979 and Bardos, Le
  Roux and Nedelec \cite{bln} that Dirichlet boundary conditions
imposed to scalar conservation laws should be understood in a
generalized sense. This can be seen by studying the parabolic
regularization of the problem. A boundary layer analysis can be
performed for systems if the solution of the conservation law is
smooth; see for instance \cite{gs,gues}. Depending on the fact that
the boundary is characteristic or not, the error is $h^{\frac12}$
or $h$.  In the scalar case, it is proved in \cite{div} that the
error between the solution of the regularized equation with a
vanishing viscosity coefficient equal to $h$ and the entropy
solution of the conservation law (which is merely of bounded variation
in space) is of order $h^{1/3}$ (in $L^\infty_t L^1_x$ norm). In
\cite{ov}, the authors derive error estimates for finite volume
schemes associated with such boundary value problems and prove that it
is of order $(\Delta x)^{1/6}$ (in $L^1_{t,x}$ norm).  More recently,
scalar conservation laws with flux constraints were studied
\cite{cg,cgs} and some finite volume schemes were built \cite{ags}.
In \cite{canseg}, assuming that the flux is bell-shaped, that is to
say the opposite is quasi-convex, it is proved that the error between
the finite volume scheme and the entropy solution is of order
$(\Delta x)^{\frac13}$ and that it can be improved to
$(\Delta x)^{\frac12}$ under an additional condition on the traces of
the BV entropy solution. It is not known if the
  estimates from \cite{canseg} are optimal or not.

On the other hand, the derivative of a viscosity solution of a
Hamilton-Jacobi equation posed on the real line is known to coincide
with the entropy solution of the corresponding scalar conservation
law.  It is therefore reasonable to expect that the error between the
viscosity solution of the Hamilton-Jacobi equation and its
approximation is as good as the one obtained between the entropy
solution of the scalar conservation law and its approximation.

Moreover, it is explained in \cite{IMZ} that the junction conditions
of optimal-control type are related to the BLN condition mentioned
above; such a correspondance is recalled in Appendix~\ref{app:bln}.
It is therefore interesting to get an error estimate of order $(\Delta
x)^{1/2}$ for the Hamilton-Jacobi problem.

\subsection{Open problems}

Let us first mention that it is not known if the error estimate
between the (entropy) solution of the scalar conservation law with
Dirichlet boundary condition and the solution of the parabolic
approximation \cite{div} or with the numerical scheme \cite{ov} is
optimal or not.
Here, we prove an optimal error estimate for $A>A_{0}$ but we do not know if our error estimate is optimal or not for $A=A_{0}$.

Deriving error estimates for general junction conditions seems
difficult to us. The main difficulty is the singular geometry of the
domain. The test function, used in deducing the error estimates
with flux limited solutions, is designed to compare flux limited
solutions. Consequently, when applying the reasoning of
Section~\ref{sec:ee}, the discrete viscosity inequality cannot be
combined with the continuous one. We expect that a layer develops
between the continuous solution and the discrete scheme at the
junction point. 

\paragraph{Organization of the article.} The remaining of the paper is
organized as follows. In Section~\ref{sec:prelim}, we recall
definitions and results from \cite{IM} about viscosity solutions for
\eqref{eq:main}-\eqref{eq:ic}.
Section~\ref{sec:grad} is dedicated to the derivation of discrete
gradient estimates for the numerical scheme. In
Section~\ref{sec:conv}, the convergence result, Theorem~\ref{Conv-gen}
is proved. In Section~\ref{sec:optimaltraject}, we study the recuced the minimal action for a ``strictly'' limited flux and prove that the gradient
is locally Lipschitz continuous (at least if the flux is strictly
limited). The final section, Section~\ref{sec:ee}, is dedicated to the
proof of the error estimates. 

\section{Preliminaries}
\label{sec:prelim}

\subsection{Viscosity solutions}

We introduce the main definitions related to viscosity solutions for
Hamilton-Jacobi equations that are used in the remaining. For a more
general introduction to viscosity solutions, the reader could refer to
Barles \cite{B} and to Crandall, Ishii, Lions \cite{CIL}.

\subparagraph{Space of test functions.} For a smooth real valued
function $u$ defined on $J$, we denote by $u^{\alpha}$ the restriction
of $u$ to $(0,T)\times J_{\alpha}$.

Then we define the natural space of functions on the junction:
\[C^1(J_T)=\{u \in C(J_T) : \forall \alpha =1,\dots,N, u^\alpha \in
C^1((0,T)\times J_\alpha )\}.\]

\subparagraph{Viscosity solutions.} In order to define classical
viscosity solutions, we recall the definition of upper and lower
semi-continuous envelopes $u^{\star}$ and $u_{\star}$ of a (locally
bounded) function $u$ defined on $[0, T ) \times J$:
\[u^{\star}(t,x)=\limsup_{(s,y)\to(t,x)}u(s,y)\qquad u_{\star}(t,x)=\liminf_{(s,y)\to(t,x)}u(s,y).\]
\begin{defi}[\textbf{Viscosity solution}]
\label{Def1}
Assume that the Hamiltonians satisfy \eqref{H} and that $F$ satisfies
\eqref{F} and let $u \colon (0,T)\times J \to \rr.$

\begin{enumerate}[label=\roman*,align=CenterWithParen]
\item We say that $u$ is a \emph{sub-solution}
  (resp. \emph{super-solution}) of (\ref{eq:main}) in $(0,T)\times J$
  if for all test function $\varphi\in C^{1}(J_{T})$ such that
\[u^{\star}\le \varphi \quad (\text{resp. }  u_{\star}\ge \varphi) 
\quad \text{ in a neighborhood of } \quad (t_0,x_0)\in J_{T}\]
 with equality at $(t_0,x_0)$ for some $t_{0}>0,$ we have 
\[\varphi_{t}+H_{\alpha}(\varphi_{x})\le 0 \quad (\text{resp. } \ge 0) \quad \text{ at } (t_{0},x_{0})\in (0,T)\times J_{\alpha} \]
if $x_0 \neq 0$, else
 \[ \varphi_{t}+F_A(\frac{\partial \varphi}{\partial_{x_1}},\dots,
\frac{\partial \varphi}{\partial_{x_N}})\le 0 \quad (\text{resp. } \ge
0) \quad \text{ at } (t_{0},x_{0})=(t_0,0). \]
\item We say that $u$ is a \emph{sub-solution}
  (resp. \emph{super-solution}) of (\ref{eq:main})-(\ref{eq:ic}) on
  $[0,T)\times J$ if additionally
\[ u^{\star}(0,x)\le u_{0}(x) \quad (\text{resp. }  u_{\star}(0,x)\ge u_{0}(x)) 
\quad \text{ for all} \quad  x\in J.\]
\item We say that $u$ is a (viscosity) solution if $u$ is both a
  sub-solution and a super-solution.
 \end{enumerate}
\end{defi}
As explained in \cite{IM}, it is difficult to construct viscosity
solutions in the sense of Definition~\ref{Def1} because of the
junction condition. It is possible in the case of the flux-limited
junction conditions $F_A$. For general junction conditions, the Perron
process generates a viscosity solution in the following relaxed sense
\cite{IM}.
\begin{defi}[\textbf{Relaxed viscosity solution}]
\label{Def2}
Assume that the Hamiltonians satisfy \eqref{H} and that $F$ satisfies
\eqref{F} and let $u \colon (0,T)\times J\to\rr.$

\begin{enumerate}[label=\roman*,align=CenterWithParen]
\item We say that $u$ is a \emph{relaxed sub-solution}
  (resp. \emph{relaxed super-solution}) of (\ref{eq:main}) in
  $(0,T)\times J$ if for all test function $\varphi\in C^{1}(J_T)$
  such that
\[ u^\star \le \varphi \quad (\text{resp. }  u_{\star} \ge \varphi) 
\quad \text{ in a neighborhood of } \quad (t_0,x_0)\in J_{T} \]
 with equality at $(t_0,x_0)$ for some $t_{0}>0,$ we have 
\[\varphi_{t}+H_{\alpha}(\varphi_{x})\le 0\qquad (\text{resp.}\quad \ge 0) 
\quad \text{ at } \quad (t_{0},x_{0})\in (0,T)\times J_{\alpha}\]
if $x_{0}\neq 0$, else
\[\begin{cases}
  \text{either } \quad \varphi_{t}+F(\frac{\partial \varphi}{\partial x_1},\dots,\frac{\partial \varphi}{\partial x_N})\le 0 
  \quad (\text{resp. }  \ge 0) & \text{ at } (t_{0},x_{0})=(t_0,0)\\
  \text{or } \quad \varphi_{t}+H_{\alpha}(\frac{\partial \varphi}{\partial x_\alpha})\le
  0 \quad (\text{resp. }  \ge 0) & \text{ at } (t_{0},x_{0})=(t_0,0) \quad
  \text{ for some } \alpha.
\end{cases} \]
\item We say that $u$ is a \emph{relaxed (viscosity) solution} of
  (\ref{eq:main}) if $u$ is both a sub-solution and a super-solution.
 \end{enumerate}
\end{defi}

Let us recall some theorems in \cite{IM}.
\begin{theorem}[\textbf{Comparison principle on a
    junction}]\label{Theo4}Let $A\in\rr\cup\{-\infty\}$. Assume that
  the Hamiltonians satisfy \eqref{H} and the initial datum $u_0$ is
  uniformly continuous. Then for all sub-solution $u$ and
  super-solution $v$ of \eqref{eq:continuous}-\eqref{eq:ic}
  satisfying for some $T>0$ and $C_T>0$
  \[ u(t,x)\le C_{T}(1+d(0,x)), \quad v(t,x)\ge -C_{T}(1+d(0,x)),
  \quad \text{ for all} \quad (t,x)\in [0,T)\times J,\] we
  have 
\[u\le v \quad in\quad [0,T)\times J.\]
\end{theorem}
\begin{theorem}[\textbf{General junction conditions reduce to
    flux-limited ones}] \label{Theo5} Assume that the Hamiltonians
  satisfy \eqref{H} and that $F$ satisfies \eqref{F}. Then there exists
  $A_F\in\rr$ such that any relaxed viscosity (sub-/super-)solution of
  \eqref{eq:main} is in fact a viscosity (sub-/super-)solution of
  \eqref{eq:continuous} with $A=A_F.$
\end{theorem}
\begin{theorem}[\textbf{Existence and uniqueness on a
    junction}] \label{thm:exis-cont} Assume that the Hamiltonians satisfy
  \eqref{H} and that $F$ satisfies \eqref{F} and that the initial datum
  $u_0$ is Lipschitz continuous. Then there exists a unique relaxed
  viscosity solution $u$ of \eqref{eq:main}-\eqref{eq:ic}, such that
\[\abs{u(t,x)-u_{0}(x)}\le C t \quad \text{ for all } \quad (t,x)\in[0,T)\times J\]
for some constant $C$ only depending on $H$ and $u_0$. Moreover, it is
Lipschitz continuous with respect to time and space, in particular,
\[ \| \nabla u \|_\infty \le C.\]
\end{theorem}
The following proposition is a main tool in the proof of error estimates. Indeed, we use a test function which is not $C^{1}$ with respect to the gradient variable at one point and this proposition allows us to get a ``weak viscosity inequality''. 
We don't give the proof since it is the same as the proof of \cite[Proposition 2.16]{IM}. 

\begin{proposition}[Non $C^{1}$ test function at one point \cite{IM}]
\label{propimp}
Assume that $H$ satisfies \eqref{H} and let $u$ be a solution of 
 $$u_t+H_{\alpha}(u_x)=0 \quad \text{ in } (0,T)\times J_\alpha \setminus \{0\}.$$
 For all $x_{0}\in J_\alpha \setminus \{0\}$ and all test function $\varphi\in C^{1}((0,T)\times J_\alpha \setminus \{0,x_{0}\})$ 
\[u^{\star}\le \varphi \quad (\text{resp. }  u_{\star}\ge \varphi) 
\quad \text{ in a neighborhood of } \quad (t_0,x_{0})\in (0,T)\times J_\alpha \setminus \{0\}\]
 with equality at $(t_0,x_{0})$, we have 
\[\varphi_{t}(t_{0},x_{0})+\max\left\{H_{\alpha}^{+}(\varphi_{x}(t_{0},x_{0}^{-}),H_{\alpha}^{-}(\varphi_{x}(t_{0},x_{0}^{+})\right\} \le 0 \quad (\text{resp. } \ge 0). \]
\end{proposition}

\section{Discrete gradient estimates}
\label{sec:grad}

This section is devoted to the proofs of the discrete (time and space)
gradient estimates. These estimates ensure the monotonicity of the
scheme and, in turn, its convergence.
The discrete time derivative is defined as 
\[ W_i^{\alpha,n}:=\frac{U_i^{\alpha,n+1}-U_i^{\alpha,n}}{\Delta t}.\]

\begin{theorem}[\textbf{Discrete gradient estimates}] \label{Theo1} If
  $u^h=(U_i^{\alpha,n})$ is the numerical solution of
  \eqref{scheme}-\eqref{eq:cid} and if the CFL condition \eqref{cflr} is
  satisfied and if
{\begin{equation}
\label{maro}
m^0=  \inf_{\substack {\beta=1,\dots,N,\\i \in \nn} }W_i^{\beta,0}
\end{equation} 
is finite}, then the following two properties hold
  true for any $n\ge 0$.
\begin{enumerate}[label=\roman*,align=CenterWithParen]
\item (Gradient estimate) There exist
  $\underline{p}_{\alpha}$,$\overline{p}^{\alpha}$,
  $\underline{p}^0_{\alpha}$ (only depending on $H_\alpha$, $u_0$ and
  $F$) such that 
\begin{equation}
\label{gradient}
\left\lbrace\begin{array}{l}
\underline{p}_{\alpha}\le p_{i,+}^{\alpha,n}\le \overline{p}^{\alpha} \qquad i\ge 1, \; \alpha=1,\dots,N, \medskip\\ 
\underline{p}_{\alpha}^0\le p_{0,+}^{\alpha,n}\le \overline{p}^{\alpha}\qquad i= 0, \; \alpha=1,\dots,N.
\end{array}\right.
\end{equation}
\item (Time derivative estimate) 
The discrete time derivative $W_i^{\alpha,n}$
satisfies
\[m^0\le m^n\le m^{n+1}\le M^{n+1}\le M^n\le M^0\]
where
\[m^n :=\inf_{\alpha, i} W_{i}^{\alpha, n}, \qquad  M^n :=\sup_{\alpha, i} W_{i}^{\alpha, n}.\]
\end{enumerate}
\end{theorem}
{In the proofs of discrete gradient estimates, ``generalized'' inverse
functions of $H_{\alpha}^{\pm}$ are needed; they are defined  as follows:
\begin{equation}
\label{inverse}
\begin{cases}
\pi_\alpha^+ (a):=\sup \{ p : H_{\alpha}^+ (p) = \max(a, A_{\alpha}) \} \\
\pi_\alpha^- (a):=\inf \{ p : H_{\alpha}^- (p) = \max(a,A_{\alpha}) \}
\end{cases}
\end{equation}
with the additional convention that $(H_{\alpha}^{\pm})^{-1}(+\infty)=\pm \infty$, where  
\[A_{\alpha}:=\min_{\substack{\rr}} H_{\alpha}.\]

In order to define a ``generalized'' inverse function of $F$, we
remark that \eqref{F} implies that
\[ \text{ for all } K \in \rr, \text{ there exists } \underline{\rho}(K)=(\rho_{1}(K),\dots,\rho_{N}(K)) \in\rr^N \text{
  such that } F(p_1,\dots,p_N)\le K \Rightarrow p_{\alpha}\ge
\underline{\rho}_{\alpha}(K).\]
Remark that the functions $\underline{\rho}_\alpha$ can be chosen non-increasing. }
\begin{remark}
  The quantities $\underline{p}_{\alpha}$,$\overline{p}^{\alpha}$,
  $\underline{p}^0_{\alpha}$ are defined as follows
\begin{equation}
\label{grd}
\left\lbrace
\begin{array}{l}
\underline{p}_{\alpha}= \left\lbrace
\begin{array}{ll}
\pi_\alpha^-(-m^0) & \text{if}-m_{0}>A_{\alpha}\\
\pi_\alpha^-(-m^0+1) & \text{if}-m_{0}=A_{\alpha}
\end{array}
\right.
 \medskip \\ 
\overline{p}_{\alpha}=  \left\lbrace
\begin{array}{ll}
\pi_\alpha^+(-m^0) & \text{if}-m_{0}>A_{\alpha}\\
\pi_\alpha^+(-m^0+1) & \text{if}-m_{0}=A_{\alpha}
\end{array}
\right.\medskip \\
\underline{p}^0_\alpha = \left\lbrace
\begin{array}{ll}
  \underline{\rho}_\alpha(-m^0) & \text{ if } \underline{\rho}_\alpha(-m^0)<\overline{p}_\alpha\\
  \underline{\rho}_\alpha (-m^0+1) & \text{ if } \underline{\rho}_\alpha(-m^0)=\overline{p}_\alpha
\end{array}
\right.
\end{array}
\right.
\end{equation}
where $m^0$ is defined in \eqref{maro}.
\end{remark}
In order to establish Theorem \ref{Theo1}, we first prove two
auxiliary results. In order to state them, some notation should be introduced.

\subsection{Discrete time derivative estimates}

 In order to state the first one,
Proposition~\ref{Prop2} below, we introduce some notation.  For
$\sigma \in \{+,-\}$, we set
\[I_{i,\sigma}^{\alpha, n}:= [\min (p_{i,\sigma}^{\alpha,n},p_{i,\sigma}^{\alpha,n+1}),
\max (p_{i,\sigma}^{\alpha,n},p_{i,\sigma}^{\alpha,n+1})] \]
with $p_{i,\sigma}^{\alpha,n}$ defined in \eqref{dsg} and 
\begin{equation}
\label{intervalle}
D_{i,+}^{\alpha,n}:=\sup\bigg\{\sup_{p_{\alpha}\in I_{i,+}^{\alpha,n}}\abs{H'_{\alpha}(p_{\alpha})},
\sup_{p_{\alpha}\in I_{0,+}^{\alpha,n}}\bigg\{-(\nabla \cdot F)(p_1,\dots,p_N)\bigg\}\bigg\}.
\end{equation}
The following proposition asserts that if the discrete space gradients enjoy
suitable estimates, then the discrete time derivative is controlled. 
\begin{proposition}[\textbf{Discrete time derivative estimate}]\label{Prop2}
  Let $n\ge 0$ be fixed and $\Delta x$, $\Delta t > 0.$ Let us
  consider $(U_{i,\alpha}^{\alpha,n})_{\alpha,i}$ satisfying for some
  constant $C^n >0:$
\[\abs{p_{i,+}^{\alpha,n}} \le C^n \qquad \text{for} \quad i\ge 0, \; \alpha=1,\dots,N.\]
We also consider $(U_{i}^{\alpha,n+1})_{\alpha,i}$ and $(U_{i}^{\alpha,n+2})_{\alpha,i}$ computed using the scheme \eqref{scheme}.
If 
\begin{equation}
\label{d}
D_{i,+}^{\alpha,n}\le \frac{\Delta x}{\Delta t}\qquad \text{ for  } \quad i\ge 0, \;\alpha=1,\dots,N,
\end{equation}
then 
\[m^n\le m^{n+1}\le M^{n+1}\le M^n.\]
\end{proposition}
\begin{proof}[\textbf{Proof}]
For $\sigma =+$ (resp. $\sigma=-$), $-\sigma$ denotes $-$ (resp. $+$).
We introduce for $n\ge 0,$ $\alpha \in \{ 1,\dots,N\}$, $i\in \{1,\dots,N\}$, $\sigma\in\{+,-\}$,
\begin{eqnarray}
\label{a}
   C_{i,\sigma}^{\alpha,n}& :=&-\sigma\int_{0}^{1}(H^{-\sigma}_{\alpha})'
\left(p_{i,\sigma}^{\alpha,n+1}+\tau(p_{i,\sigma}^{\alpha,n}-p_{i,\sigma}^{\alpha,n+1})\right)d\tau
\ge 0, \\
\nonumber  {C_{0,+}^{\alpha,n}}& :=&-\int_{0}^{1}\frac{\partial F}{\partial p_{\alpha}}
\bigg(\{p_{0,+}^{\beta,n+1}+\tau(p_{0,+}^{\beta,n}-p_{0,+}^{\beta,n+1})\}_{\beta}\bigg)d\tau\ge 0.
\end{eqnarray}
Notice that for $i\ge 1$, $C_{i,\sigma}^{\alpha,n}$ is defined as the
integral of $(H_{\alpha}^{-\sigma})'$ over a convex combination of
$p\in I_{i,\sigma}^{\alpha,n}.$ Similarly for $C_{0,+}^{\alpha,n}$
which is defined as the integral of $F'$ on a convex combination of
$p\in I_{0,+}^{\alpha,n}.$ Hence, in view of \eqref{d}, we have for any $n\ge 0$, $\alpha=1,\dots,N$
and for any $\sigma \in\{+,-\}$ or for $i=0$ and $\sigma=+,$
we can check that
\begin{equation}
\label{c}
\begin{cases} C_{i,\sigma}^{\alpha,n}\le \frac{\Delta x}{\Delta t} & \text{ if } i \ge 1, \, \sigma \in \{-,+\} \medskip \\
\sum_{\beta =1}^N C_{0,+}^{\beta,n} \le \frac{\Delta x}{\Delta t}. & 
\end{cases}
\end{equation}

We can also underline that for any $n\ge 0$, $\alpha=1,\dots,N$ and
for any $i\ge 1,$ $\sigma \in\{+,-\}$ or for $i=0$ and $\sigma=+,$ we
have the following relationship
\begin{equation}
\label{b}
\frac{p_{i,\sigma}^{\alpha,n}-p_{i,\sigma}^{\alpha,n+1}}{\Delta t}=
-\sigma\frac{W_{i+\sigma}^{\alpha,n}-W_{i}^{\alpha,n}}{\Delta x}.
\end{equation}
Let $n\ge 0$ be fixed and consider $(U_i^{\alpha,n})_{\alpha,i}$ with
$\Delta x,\Delta t>0$ given. We compute
$(U_i^{\alpha,n+1})_{\alpha,i}$ and $(U_i^{\alpha,n+2})_{\alpha,i}$
using the scheme \eqref{scheme}.  \medskip

\paragraph{Step 1: $(m^n)_n$ is non-decreasing.} We want to show that
$W_i^{\alpha,n+1}\ge m^n$ for $i\ge 0$ and $\alpha=1,\dots,N.$ Let
$i\ge 0$ be fixed and let us distinguish two cases.

\subparagraph{Case 1: $i\ge 1$.}
Let a branch $\alpha$ be fixed and let  $\sigma(i,\alpha,n+1)=\sigma \in\{+,-\}$ be such that
\begin{equation}
\label{m}
\max\bigg\{H_{\alpha}^{+}(p_{i,-}^{\alpha,n+1}),H_{\alpha}^{-}(p_{i,+}^{\alpha,n+1})\bigg\}
=H_{\alpha}^{-\sigma}(p_{i,\sigma}^{\alpha,n+1}).
\end{equation}
We have
\[\begin{split}
\frac{W_i^{\alpha,n+1}-W_i^{\alpha,n}}{\Delta t} & =\frac{1}{\Delta t}\bigg(\max\bigg\{H_{\alpha}^{+}(p_{i,-}^{\alpha,n}),H_{\alpha}^{-}(p_{i,+}^{\alpha,n})\bigg\}
-\max\bigg\{H_{\alpha}^{+}(p_{i,-}^{\alpha,n+1}),H_{\alpha}^{-}(p_{i,+}^{\alpha,n+1})\bigg\}\bigg)\\
& \ge\frac{1}{\Delta t}\bigg(H_{\alpha}^{-\sigma}(p_{i,\sigma}^{\alpha,n})- H_{\alpha}^{-\sigma}(p_{i,\sigma}^{\alpha,n+1})\bigg)\\
&=\int_{0}^{1}(H_{\alpha}^{-\sigma})'(p_{i,\sigma}^{\alpha,n+1}+\tau (p_{i,\sigma}^{\alpha,n}-p_{i,\sigma}^{\alpha,n+1}))
\left(\frac{p_{i,\sigma}^{\alpha,n}-p_{i,\sigma}^{\alpha,n+1}}{\Delta t}\right) d\tau \\
&=C_{i,\sigma}^{\alpha,n}\bigg( \frac{W_{i+\sigma}^{\alpha,n}-W_{i}^{\alpha,n}}{\Delta x}\bigg)
\end{split}\]
where we used \eqref{a} and \eqref{b} in the last line. Using \eqref{c}, we thus get
\[\begin{split}
W_{i}^{\alpha,n+1}&\ge \bigg(1-C_{i,\sigma}^{\alpha,n}\frac{\Delta t}{\Delta x}\bigg)
W_{i}^{\alpha,n}+C_{i,\sigma}^{\alpha,n}\frac{\Delta t}{\Delta x}W_{i+\sigma}^{\alpha,n}\\
&\ge m^{n}.
\end{split}\]
 
\subparagraph{Case 2: $i=0$.}  We recall that in this case, we have
$U_0^{\beta,n}:=U_0^{n}$ and
$W_0^{\beta,n}:=W_0^n=\frac{U_0^{n+1}-U_0^n}{\Delta t}$ for any
$\beta=1,\dots,N.$ We compute in this case:
\[\begin{split}
\frac{W_0^{n+1}-W_0^{n}}{\Delta t}&=\frac{1}{\Delta t}\left(-F(\{p_{0,+}^{\alpha,n+1}\}_{\alpha})+F(\{p_{0,+}^{\alpha,n}\}_{\alpha})\right)  \\
&= \frac{1}{\Delta t}\int_0^1\sum_{\beta=1}^N p_\beta \frac{\partial F}{\partial p_{\beta}}\left(\{p_{0,+}^{\alpha,n+1}+\tau p_\alpha \}_{\alpha}\right)d\tau \qquad 
\text{with} \,\,\, p=(\{p_{0,+}^{\alpha,n}-p_{0,+}^{\alpha,n+1}\}_{\alpha})\\
&=-\int_0^1\sum_{\beta=1}^N\frac{\partial F}{\partial p_{\beta}}\left(\{p_{0,+}^{\alpha,n+1}+\tau p_\alpha\}_{\alpha}\right) d\tau
\bigg(\frac{W_1^{\beta,n}-W_0^n}{\Delta x}\bigg)\\
&=\sum_{\beta=1}^N C_{0,+}^{\beta,n}\bigg(\frac{W_1^{\beta,n}-W_0^n}{\Delta x}\bigg).
\end{split}\]
Using \eqref{c}, we argue like in Case~1 and get 
\[W_0^{n+1}\ge m^n.\]

\paragraph{Step 2: $(M^n)_n$ is non-increasing.}  
We want to show that $W_i^{\alpha,n+1}\le M^n$ for $i\ge 0$ and
$\alpha=1,\dots,N.$ We argue as in Step~1 by  distinguishing two cases.

\subparagraph{Case 1: $i\ge 1$.} We simply choose 
{$\sigma=\sigma(i,\alpha,n)$ (see \eqref{m})}
and argue as in Step~1.

\subparagraph{Case 2: $i=0$.} Using \eqref{d}, we can argue exactly as in Step~1. 
The proof is now complete. 
\end{proof}

\subsection{Gradient estimates}

The second  result needed in the proof of Theorem \ref{Theo1}
is the following one. It asserts that if the discrete time derivative
is controlled from below, then a discrete gradient estimate holds true. 
\begin{proposition}[\textbf{Discrete gradient estimate}]\label{Prop11}
  Let $n\ge 0$ be fixed, consider that $(U_i^{\alpha,n})_{\alpha,i}$
  is given and compute $(U_i^{\alpha,n+1})_{\alpha,i}$ using the
  scheme \eqref{scheme}-\eqref{dsg}.  If there exists a constant $K\in \rr$ such
  that for any $i\ge 0$ and $\alpha=1,\dots, N,$
\[K\le W_i^{\alpha,n}:=\frac{U_i^{\alpha,n+1}-U_i^{\alpha,n}}{\Delta t}\] 
then 
\[\left \lbrace
\begin{array}{rcll}
\pi_\alpha^- (-K) \le & p_{i,+}^{\alpha,n} & \le \pi_\alpha^+ (-K), &  \alpha=1,\dots,N, \quad i\ge 1,  \medskip \\ 
\underline{\rho}_{\alpha}(-K) \le & p_{0,+}^{\alpha,n} &\le (H_{\alpha}^{+})^{-1}(-K), &  \alpha=1,\dots,N
\end{array}
\right.
\]
where $p_{i,+}^{\alpha,n}$ is defined in \eqref{dsg} and
$\pi_\alpha^\pm$ and $\underline{p}$ are the ``generalized'' inverse
functions of $H_\alpha$ and $F$, respectively.
\end{proposition}
\begin{proof}[\textbf{Proof}]
  Let $n\ge 0$ be fixed and consider $(U_i^{\alpha,n})_{\alpha,i}$
  with $\Delta x,\Delta t>0$ given. We compute
  $(U_i^{\alpha,n+1})_{\alpha,i}$ using the scheme~\eqref{scheme}.
 Let us consider any $i\ge 0$ and $\alpha=1,\dots,N.$ 

If $i \ge 1$,  the result follows  from
\[K\le W_i^{\alpha,n}=-\max_{\sigma = +,-} H_{\alpha}^{\sigma}(p_{i,\sigma}^{\alpha,n}).\]

If $i=0$,  the results follows  from 
\[K\le W_0^n=-F\bigg(\{p_{0,+}^{\alpha,n}\}_\alpha\bigg).\] 
This achieves the proof of Proposition~\ref{Prop11}
\end{proof}

\subsection{Proof of gradient estimates}

\begin{proof}[\textbf{Proof of Theorem~\ref{Theo1}}] 
  The idea of the proof is to introduce new Hamiltonians
  $\tilde{H_{\alpha}}$ and a new junction function $\tilde{F}$ for
  which it is easier to derive gradient estimates but whose
  corresponding numerical scheme in fact coincide with the original
  one.

 \paragraph{Step 1: Modification of the Hamiltonians and the
  junction function.}  Let the new Hamiltonians $\tilde{H_{\alpha}}$
for all $\alpha=1,\dots,N$ be defined as
\begin{equation}
\label{eq:modified-h}
\tilde{H_{\alpha}}(p)=\begin{cases}
H_{\alpha}(\underline{p}_{\alpha})-\frac{C_{\alpha}}{2}(p-\underline{p}_{\alpha}) & \text{ if }   p\le \underline{p}_{\alpha}\\
H_{\alpha}(p) & \text{ if }  p\in[\underline{p}_{\alpha}\overline{p}_\alpha]\\
H_{\alpha}(\underline{p}_{\alpha})+\frac{C_{\alpha}}{2}(p-\overline{p}_\alpha) & \text{ if }  p\ge \overline{p}_\alpha
\end{cases}
\end{equation}
where $\underline{p}_{\alpha}$ and $\overline{p}_\alpha$ are defined
in \eqref{grd} respectively, and 
\[C_{\alpha}=
\sup_{\substack{p_{\alpha}\in[\underline{p}_{\alpha},\overline{p}_\alpha]}}\abs{H_{\alpha}'(p_{\alpha})}.\]
{These new Hamiltonians are now globally Lipschitz
  continuous: their derivatives are bounded. More precisely, the}
$\tilde{H}_{\alpha}$ satisfy \eqref{H} and
\[ \tilde{H}_\alpha \equiv H_\alpha \text{ in }  [\underline{p}_\alpha,\overline{p}_\alpha]\]
and
\begin{equation}
\label{eq:htildeprime}
\forall p \in \rr, \quad 
\abs{\tilde{H}_{\alpha}'(p)} \le \sup_{\substack{p_{\alpha}\in[\underline{p}_{\alpha},\overline{p}_\alpha]}}
\abs{H_{\alpha}'(p_{\alpha})}. 
\end{equation}  
Let the new $\tilde{F}$ satisfy \eqref{F}, be such that 
\[ \tilde F \equiv F \text{ in } Q_0:=\prod_{\alpha=1}^{N}[\underline{p}_{\alpha}^0,\overline{p}_\alpha] \]
and (See Appendix \ref{AppendixB})
\begin{equation}
\label{eq:Ftildediv}
 \forall p \in \rr^N, \quad 
(-\nabla \cdot \tilde F)(p) \le \sup_{Q_0} (-\nabla \cdot F). 
\end{equation}
{In the remainder of the proof, when notation contains
  a tilde, it is associated with the new Hamiltonians
  $\tilde H_\alpha$ and the new non-linearity $\tilde F$.}  We then
consider the new numerical scheme
\[\left\lbrace
\begin{array}{lll}
\frac{\tilde{U}^{\alpha,n+1}_i-\tilde{U}^{\alpha,n}_i}{\Delta t}
+\max\{\tilde{H}_{\alpha}^{+}(\tilde{p}_{i,-}^{\alpha,n}), \tilde{H}_{\alpha}^{-}(\tilde{p}_{i,+}^{\alpha,n})\}=0, & i\ge 1, &
\alpha=1,\dots,N \medskip \\ 
\tilde{U}_0^{\beta,n}:=U_0^n, &  i=0, & \beta=1,\dots,N \medskip \\
\frac{\tilde{U}_0^{n+1}-\tilde{U}_0^{n}}{\Delta t}+\tilde{F}(\tilde{p}_{0,+}^{1,n},\tilde{p}_{0,+}^{2,n},\dots,\tilde{p}_{0,+}^{N,n})=0 &   &
\end{array}
\right.\]
with the same initial condition, namely,
\[\tilde{U}_i^{\alpha,0}=u_0^{\alpha}(i\Delta x), \quad i\ge 0,\quad \alpha=1,\dots,N.\]

In view of  \eqref{eq:htildeprime} and \eqref{eq:Ftildediv}, the CFL
condition~\eqref{cflr} gives that for any $i\ge 0$, $n\ge 0$, and
$\alpha=1,\dots,N$
\begin{equation}
\label{eq:dtilde}
\tilde{D}_{i,+}^{\alpha,n}\le \sup \left\{\sup_{\underline{p}_{\alpha}\le p\le \overline{p}_\alpha}{\abs{H_{\alpha}'(p)}}; 
\sup_{ \tilde{I}^{\alpha,n}_{0,+}} (-\nabla \cdot F) \right\}\le \frac{\Delta x}{\Delta t}
\end{equation}
{ where $\tilde{D}_{i,+}^{\alpha,n}$ is given by \eqref{intervalle} after replacing $H_\alpha$ and $F$ with $\tilde H_\alpha$ and $\tilde F$.}

\paragraph{Step 2: First gradient bounds.}
Let $n\ge 0$ be fixed. If $\tilde{m}^n$ and $\tilde{M}^n$ are finite, we have
\[\tilde{m}^n\le \tilde{W}_i^{\alpha,n}\qquad \text{for}\quad \text{any} \quad i\ge0, \quad \alpha=1,\dots,N .\] 
Proposition~\ref{Prop11} implies that
\[\left\lbrace
\begin{array}{rcll}
\tilde{\pi}_\alpha^- (-\tilde{m}^n) \le & \tilde{p}_{i,+}^{\alpha,n} & \le \tilde{\pi}_\alpha^+ (-\tilde{m}^n), &
i\ge 1,\quad \alpha=1,\dots,N, \\
\tilde{\underline{\rho}}_\alpha (-\tilde{m}^n) \le & \tilde{p}_{0,+}^{\alpha,n} & \le \tilde{\pi}_\alpha^+ (-\tilde{m}^n), &  i\ge 0,\quad \alpha=1,\dots,N.
\end{array}
\right.\]
In particular,  we get that
\[\abs{\tilde{p}_{i,+}^{\alpha,n}}\le C^n \quad \text{ for } \quad i\ge 0,\quad \alpha=1,\dots,N\]
with
\[C^n=\max_{\alpha} \left(\max \left( |\tilde{\pi}_\alpha^- (-\tilde{m}^n)|,| \tilde{\pi}_\alpha^+ (-\tilde{m}^n)|,
|\tilde{\underline{\rho}}_\alpha (-\tilde{m}^n)| \right)\right).\]
In view of \eqref{eq:dtilde}, Proposition~\ref{Prop2} implies that 
\begin{equation}
\label{ali}
\tilde{m}^n\le\tilde{m}^{n+1}\le\tilde{M}^{n+1}\le\tilde{M}^{n}\,\quad \text{for} \quad \text{any}\hskip 0.2cm n\ge 0.
\end{equation}
In particular, $\tilde{m}^{n+1}$ is also finite. Since
$\tilde{m}^0=m^0$ and $\tilde{M}^0 {=M^0}$ are finite, we conclude that
$\tilde{m}^n$ and $\tilde{M}^n$ are finite for all $n \ge 0$ and for
all $n \ge 0$,
\begin{equation}
\label{eq:time-der-tilde}
 m^0  \le \tilde{m}^n \le \tilde{M}^n  \le M^0.
\end{equation}

\paragraph{Step 3: Time derivative and gradient estimates.}
Now we can repeat the same reasoning  but applying Proposition~\ref{Prop11} with $K=m^0$ and get 
\begin{equation}
\label{eq:estim-ptilde}
\left\lbrace
\begin{array}{rcll}
\underline{p}_{\alpha}  \le & \tilde{p}_{i,+}^{\alpha,n} & \le \overline{p}_\alpha, &
i\ge 1,\quad \alpha=1,\dots,N, \\
\underline{p}_{\alpha}^0 \le & \tilde{p}_{0,+}^{\alpha,n} & \le \overline{p}_\alpha, &  i\ge 0,\quad \alpha=1,\dots,N.
\end{array}
\right.
\end{equation}
This implies that $\tilde{U}_i^{\alpha,n} = U_i^{\alpha,n}$ for all $i
\ge 0$, $n \ge 0$, $\alpha = 1,\dots, N$. In view of \eqref{ali},
\eqref{eq:time-der-tilde} and \eqref{eq:estim-ptilde}, the proof is
now complete.
\end{proof}

\section{Convergence for general junction conditions}
\label{sec:conv}

This section is devoted to the convergence of the scheme defined by
\eqref{scheme}-\eqref{dsg}. In order to do so, we first make precise
how to choose $\overline{p}_\alpha, \underline{p}_\alpha$ and
$\underline{p}_\alpha^0$ in the CFL condition~\eqref{cflr}.

\subsection{Monotonicity of the scheme}

In order to prove the convergence of the numerical solution as the
mesh size tends to zero, we need first to prove a monotonicity result.
It is common to write the scheme defined by \eqref{scheme}-\eqref{dsg}
under the compact form 
\[ u^h (t+\Delta t,x) = S^h [u^h(t)](x)\] where the operator
$S^h$ is defined on the set of functions defined in
$J^h$. The scheme is monotone if 
\[u \le v \quad \Rightarrow \quad S^h [u] \le S^h [v].\] In our
cases, if $t = n \Delta t$ and $x = i\Delta x \in J^\alpha$ and
$U(t,x) = U^{\alpha,n}_i$ for $x \in J^\alpha$, then $S^h[U]$ is
defined as follows,
\[\begin{cases}
U_i^{\alpha,n+1}=S_{\alpha}[U_{i-1}^{\alpha,n},U_i^{\alpha,n},U_{i+1}^{\alpha,n}]  \text{ for }  i\ge 1, \,\alpha=1,\dots,N, \\
U_0^{n+1}=S_{0}[U_0^{n},(U_1^{\beta,n})_{\beta=1,\dots,N}] 
\end{cases}\]
where 
\begin{equation}
\label{Salpha}
\left\lbrace
\begin{array}{l}
S_{\alpha}[U_{i-1}^{\alpha,n},U_i^{\alpha,n},U_{i+1}^{\alpha,n}]:=U_i^{\alpha,n}-\Delta t 
\max \left\{H_{\alpha}^{+}\left(\frac{U_i^{\alpha,n}-U_{i-1}^{\alpha,n}}{\Delta x}\right),
H_{\alpha}^{-}\left(\frac{U_{i+1}^{\alpha,n}-U_{i}^{\alpha,n}}{\Delta x}\right)\right\}, \medskip \\ 
S_{0}[U_0^{n},(U_1^{\beta,n})_{\beta=1,\dots,N}]:=U_0^{n}-\Delta t F(p_{0,+}^{1,n},\dots,p_{0,+}^{N,n})
\end{array}\right.
\end{equation}
Checking the monotonicity of the scheme reduces to checking that
$S_{\alpha}$ and $S_0$ are non-decreasing in all their variables.
\begin{lemma}[\textbf{Monotonicity of the numerical scheme}]\label{lem:monotone}
  Let $(U^n):=(U_i^{\alpha,n})_{\alpha,i}$ the numerical solution of
  \eqref{scheme}-\eqref{eq:cid}. Under the CFL condition \eqref{cfl}
  the scheme is monotone.
\end{lemma}
\begin{proof}[\textbf{Proof}]
We distinguish two cases. 

\paragraph{Case 1: $i\ge 1$.}
It is straightforward to check that, for any $\alpha=1,\dots,N,$ the
function $S_{\alpha}$ is non-decreasing with respect to
$U^{\alpha,n}_{i-1}$ and $U_{i+1}^{\alpha,n}$. Moreover,
\[\frac{\partial S_{\alpha}}{\partial U_i^{\alpha,n}}=\left\lbrace
\begin{array}{l}
1-\frac{\Delta t}{\Delta x}(H_{\alpha}^{+})'(p_{i,-}^{\alpha,n}) \quad \text{ if } \quad
\max\{H_{\alpha}^{+}(p_{i,-}^{\alpha,n}),H_{\alpha}^{-}(p_{i,+}^{\alpha,n})\}=H_{\alpha}^{+}(p_{i,-}^{\alpha,n}) \medskip \\
1+\frac{\Delta t}{\Delta x}(H_{\alpha}^{-})'(p_{i,+}^{\alpha,n})\quad \text{ if } \quad
\max\{H_{\alpha}^{+}(p_{i,-}^{\alpha,n}),H_{\alpha}^{-}(p_{i,+}^{\alpha,n})\}=H_{\alpha}^{-}(p_{i,+}^{\alpha,n})
\end{array}
\right.\] 
which is non-negative if the CFL condition (\ref{cfl}) is
satisfied.

\paragraph{Case 2: $i=0$.}
Similarly it is straightforward to check that $S_0$ is non-decreasing
with respect to $U_1^{\beta,n}$ for $\beta=1,\dots,N$. Moreover,
\[\frac{\partial S_0}{\partial U_0^n}=1+\frac{\Delta t}{\Delta x}\sum_{\beta=1}^{N}
\frac{\partial F}{\partial p_{\beta}}\{ (p_{0,+}^{\alpha,n})_{\alpha=1}^{N}\}\] 
which is non-negative due to the CFL condition. The proof is now complete. 
\end{proof}
A direct consequence of the previous lemma is the following elementary
but useful discrete comparison principle. 
\begin{lemma}[\textbf{Discrete Comparison Principle}]\label{Prop1}
  Let $(U^n):=(U_i^{\alpha,n})_{\alpha,i}$ and
  $(V^n):=(V_i^{\alpha,n})_{\alpha,i}$ be such that 
\[ \forall n \ge 1, \quad  U^{n+1} \le S^h [U^n] \quad \text{ and } \quad V^{n+1} \ge S^h [V^n].\]
If the CFL condition (\ref{cfl}) is satisfied
  and if $U^0 \le V^0$, then $U^n\le V^n$ for all $n \in \nn$. 
\end{lemma}
\begin{remark}\label{rem:sub-super-schemes}
  The discrete function $(U^n)$ (resp. $(V^n)$) can be seen as a
  \emph{sub-scheme} (resp. \emph{super-scheme}).
\end{remark}
We finally recall how to derive discrete viscosity inequalities for monotone schemes. 
\begin{lemma}[\textbf{Discrete viscosity inequalities}]\label{lem:visc-ineq}
  Let $u^h$ be a solution of \eqref{scheme}-\eqref{eq:cid} with
  $F=F_{A}$ defined in (\ref{fa}).  If $u^h-\varphi$ has a global
  maximum (resp. global minimum) on $\mathcal{G}_h$ at
  $(\overline{t}+\Delta t,\overline{x})$, then
\[\delta_t\varphi(\overline{t},\overline{x})
+\mathcal{H}(\overline{x},D_{+}\varphi(\overline{t},\overline{x}),D_{-}\varphi(\overline{t},\overline{x}))\le 0. 
\quad (\text{resp. } \ge 0)\] 
where 
\[ \mathcal{H}(x,p_+,p_-)=\begin{cases}
\max\{H_{\alpha}^{+}(p_-),H_{\alpha}^{-}(p_+)\} & \text{ if } \overline x\neq0\\
\max\{A,\max_\alpha H_{\alpha}^{-}(p^+_\alpha)\} & \text{ if } \overline{x}=0
\end{cases}\]
and 
\begin{eqnarray*}
D_{+}\varphi(\overline{t},\overline{x}) &=& \begin{cases}
\frac{1}{\Delta x}\{\varphi(\overline{t},\overline{x}+\Delta x)-\varphi(\overline{t},\overline{x})\} & \text{ if } \overline x \neq 0 \\
\left(\frac{1}{\Delta x}\{\varphi^\alpha(\overline{t},\Delta x)-\varphi^\alpha(\overline{t},0)\} \right)_\alpha & \text{ if } \overline x = 0
\end{cases}\\
D_{-}\varphi(\overline{t},\overline{x}) &=&\frac{1}{\Delta x}\{\varphi(\overline{t},\overline{x})-\varphi(\overline{t},\overline{x}-\Delta x)\}\\
\delta_t\varphi(\overline{t},\overline{x})&=&\frac{1}{\Delta t}\{\varphi(\overline{t}+\Delta t,\overline{x})-\varphi(\overline{t},\overline{x})\}.
\end{eqnarray*}
\end{lemma}

\subsection{Stability and Consistency of the scheme}

We first derive a local $L^\infty$ bound for the solution of the scheme.
\begin{lemma}[\textbf{Stability of the numerical
    scheme}]\label{lem:stab}
  Assume that the CFL condition (\ref{cflr}) is satisfied and let
  $u^h$ be the solution of the numerical scheme
  \eqref{scheme}-\eqref{eq:cid}. There exists a
  constant $C_0>0$, such that for all $(t,x) \in \mathcal{G}_h$, 
\begin{equation}\label{eq:estimuo}
\abs{u^h (t,x) -u_0 (x)}\le C_0 t.
\end{equation}
In particular, the scheme is (locally) stable.
\end{lemma}
\begin{proof}[\textbf{Proof}]
If $C_0$ large enough so that 
\[\left\lbrace
\begin{array}{lll}
C_0 +\max\{H_{\alpha}^{+}(p_{i,-}^{\alpha,0}),H_{\alpha}^{-}(p_{i,+}^{\alpha,0})\} \ge 0, & i\ge 1, &
\alpha=1,\dots,N\\ 
C_0 +F(p_{0,+}^{1,0},p_{0,+}^{2,0},\dots,p_{0,+}^{N,0}) \ge 0, & & 
\end{array}\right.\]
and 
\[\left\lbrace
\begin{array}{lll}
-C_0 +\max\{H_{\alpha}^{+}(p_{i,-}^{\alpha,0}),H_{\alpha}^{-}(p_{i,+}^{\alpha,0})\} \le 0, & i\ge 1, &
\alpha=1,\dots,N\\ 
- C_0 +F(p_{0,+}^{1,0},p_{0,+}^{2,0},\dots,p_{0,+}^{N,0}) \le 0, & & 
\end{array}\right.\]
then $\bar{U}^{\alpha,n}_i = U^{\alpha,0}_i + C_0 n \Delta t$ is a
super-scheme and $\bar{U}^{\alpha,n}_i = U^{\alpha,0}_i - C_0 n \Delta
t$ is a sub-scheme (see Remark~\ref{rem:sub-super-schemes}). The
discrete comparison principle, Lemma~\ref{Prop1}, then implies
\[ |U^{\alpha,n}_i  -U^{\alpha,0}_i | \le C_0 n \Delta t \]
which is the desired inequality. This achieves the proof. 
\end{proof}
Another condition to satisfy convergence of the numerical scheme
\eqref{scheme} towards the continuous solution of
\eqref{eq:continuous} is the consistency of the scheme (which is
obvious in our case). In the statement below, we use the short hand
notation~\eqref{short} introduced in Appendix.
\begin{lemma}[\textbf{Consistency of the numerical scheme}]\label{lem:cons}
  Under the assumptions on the Hamiltonians \eqref{H}, the finite
  difference scheme is consistent with the continuous problem
  \eqref{eq:continuous}, that is to say for any smooth function
  $\varphi(t,x),$ we have locally uniformly
\[ \frac{S^h[\varphi](s,y) -\varphi(s,y)}{\Delta t} \to H_\alpha (\varphi_x(t,x)) \quad \text{as}\quad 
\mathcal{G}_h \ni (s,y) \to (t,x)\]
if $x \in J_\alpha \setminus \{0\}$, and 
\[ \frac{S^h[\varphi](s,y) -\varphi(s,y)}{\Delta t} \to F
\left(\frac{\partial \varphi}{\partial x_{1}},\dots,\frac{\partial
  \varphi}{\partial x_{N}}(t,0)\right) \quad \text{as}\quad \mathcal{G}_h
\ni (s,y) \to (t,0).\]
\end{lemma}

\subsection{Convergence of the numerical scheme}

In this subsection, we present a sketch of the proof of Theorem \ref{Conv-gen}. 

\begin{proof}[Sketch of the proof of Theorem~\ref{Conv-gen}]
  Let $T > 0$ and $h := (\Delta t, \Delta x)$ satisfying the CFL
  condition (\ref{cflr}).  We recall that
  \[u^h (0,x) = u(0,x) \quad \text{ for } \quad x \in
  \mathcal{G}_h. \] 
  We consider $\overline{u}$ and $\underline{u}$ respectively defined
  as
\[\overline{u}(t,y)=
\limsup_{\substack{h\to 0\\\mathcal{G}_{h}\ni(t',y')\to (t,y)}}u^{h}(t',y'), 
\qquad \underline{u}(t,y)=
\liminf_{\substack{h\to 0\\\mathcal{G}_{h}\ni(t',y')\to (t,y)}}u^{h}(t',y').\]
By construction, we have $\underline{u} \le \overline{u}$.  Since the
scheme is monotone (Lemma~\ref{lem:monotone}), stable
(Lemma~\ref{lem:stab}) and consistent (Lemma~\ref{lem:cons}), we can
follow \cite{CL,BS,CLM} we can show that $\underline{u}$
(resp. $\overline{u}$) is a relaxed viscosity super-solution
(resp. viscosity sub-solution) of equation
\eqref{eq:main}-\eqref{eq:ic}, see for example the proof of \cite[Theorem 1.8]{CLM}. Using Theorem \ref{Theo5}, we know
that $\underline{u}$ (resp. $\overline{u}$) is a viscosity
super-solution (resp. sub-solution) of
\eqref{eq:continuous}-\eqref{eq:ic}. 
Moreover, \eqref{eq:estimuo} implies that 
\[ \overline{u} (0,x) \le u_0 (x) \le \underline{u} (0,x). \]
 The comparison principle (see
Theorem \ref{Theo4}) then implies that
\[\overline{u} \le u \le \underline{u}\] 
which achieves the proof. 
\end{proof}

\section{Study of the reduced minimal action}
\label{sec:optimaltraject}

In this section, we consider that the Hamiltonians $H_{\alpha}$ satisfy \eqref{Jconvex1}. We study the reduced minimal action $\mathcal{D}_{0}$ which replace the classical term $\frac{(x-y)^{2}}{2\epsilon}$ in the doubling variable method. This function allows us to prove that the error estimate is of order $(\Delta x)^{\frac{1}{2}}$.

\subsection{Reduction of the study}
We start this section by the following remark, the analysis can be reduced to the case $(s,t)=(0,1).$ Precisely, using the fact that the Hamiltonian does not depend on time and is homogeneous with respect to the state, the reader can check that a change of variables in time yields the following Lemma.
\begin{lemma}
For all $y,x \in J$ and $s<t,$ we have
\[
\mathcal{D}(s,y;t,x)=(t-s)\mathcal{D}\bigg(0,\frac{y}{t-s};1, \frac{x}{t-s}\bigg).
\]
where
\[
\mathcal{D}(s,y;t,x)=\inf_{\substack{{X\in \mathcal{A}(s,y;t,x)}}}\bigg\{\int_s^t L(X(\tau),\dot{X}(\tau))d\tau\bigg\}.
\]
\end{lemma}
This is the reason why we consider the reduced minimal action $\mathcal{D}_0: J^2\rightarrow \mathbb{R}$ defined by
\begin{equation}
\label{J20}
\mathcal{D}_0(y,x)=\mathcal{D}(0,y;1,x).
\end{equation}
We also need the following lower bound on $\mathcal{D}.$
\begin{lemma}
\label{JLem7}
Assume (B0). Then
\[
\mathcal{D}(s,y;t,x)\ge \frac{\gamma}{2(t-s)}d^2(x,y)-A(t-s)
\]
 where $\gamma$ is defined in (B0).

Moreover,
\[
\mathcal{D}(s,x;t,x)\le L_A(0)(t-s).
\]
\end{lemma}
\begin{proof}
We only prove the first inequality since the other inequality is elementary. As $L'_{\alpha}(0)=0$, and $L_{\alpha}(0)\ge L_{A}(0)=-A$, we have
\[L_{\alpha}(p)\ge \frac{\gamma}{2}p^2+L_{\alpha}^{\prime}(0)p+L_{\alpha}(0)\ge \frac{\gamma}{2}p^2-A.\]
Thus, we can write for $X(.)\in \mathcal{A}(s,y;t,x),$
\[
\int_s^tL(X(\tau),\dot{X}(\tau))d\tau\ge -A(t-s)+\frac{\gamma}{2}\int_s^t(\dot{X}(\tau))^{2}d\tau.
\]
Then Jensen's inequality allows us to conclude.
\end{proof}
\subsection{Piecewise linear trajectories}
We are going to see that the infimum defining the minimal action can be computed among piecewise linear trajectories. In order to state a precise statement, we first introduce that optimal curves are of two types depending on the position of $y$ and $x$ on the same branch or not: if they are, then the trajectories are of two types: either they reach the junction point, or they stay in a branch and are straight lines. For $y\in J_{\beta},$ $x\in J_{\alpha}$ with $\beta\neq \alpha,$ the trajectories can spend some time at the junction point.

\begin{lemma}
The infimum defining the reduced minimal action $\mathcal{D}_0$ can be computed among piecewise linear trajectories; more precisely for all $y,x\in J,$
\begin{equation}
\label{J21}
\mathcal{D}_0(y,x)=
\left\lbrace
\begin{array}{ll}
\mathcal{D}_{junction}(y,x) & \text{ if } \alpha\neq \beta,\\
\min(L_{\alpha}(x-y),\mathcal{D}_{junction}(y,x)) & \text{ if } \alpha=\beta,
\end{array}
\right.
\end{equation}
where for $x\in J_{\alpha}, y\in J_{\beta}$
\begin{equation}
\label{J25}
\mathcal{D}_{junction}(y,x)=\inf_{\substack{{0\le t_1\le t_2 \le 1}}}\bigg\{t_{1}L_{\beta}\bigg(\frac{-y}{t_1}\bigg)+(t_2-t_1)L_A(0)+(1-t_2)L_{\alpha}\bigg(\frac{x}{1-t_2}\bigg)\bigg\}.
\end{equation}
\end{lemma}
\begin{proof}
We write $\mathcal{D}_0=\inf_{\substack{{X\in \mathcal{A}_0(y,x)}}}\Lambda(X),$ where $\Lambda(X)=\int_{0}^{1}L(X(\tau), \dot{X}(\tau))d\tau.$ In order to prove the lemma, it is enough to consider a curve $X\in \mathcal{A}(0,y;1,x)$ and prove that
\[\Lambda(X)\ge \min(L_{\alpha}(x-y), D_{junction}(y,x)).\]
For $\alpha\neq \beta,$ the trajectories can spend some time at the junction point, hence we can write
\begin{equation*}
\begin{split}
\mathcal{D}_0(y,x) & = \inf_{\substack{{X(0)=y}\\{X(1)=x}}}\bigg\{\int_0^{t_1}L_{\beta}(\dot{X}(\tau))d\tau+\int_{t_1}^{t_2}L(X(\tau), \dot{X}(\tau))d\tau+\int_{t_2}^{1}L_{\alpha}(\dot{X}(\tau))d\tau\bigg\}
\\
& 
\ge \inf_{\substack{0\le t_1\le t_2\le 1}}\bigg\{\inf_{\substack{{X(0)=y}\\{X(t_1)=x}}}\int_0^{t_1}L_{\beta}(\dot{X}(\tau))d\tau+\inf_{\substack{{X(t_1)=0}\\{X(t_2)=0}}}\int_{t_1}^{t_2}L(X(\tau), \dot{X}(\tau))d\tau \\ & +\inf_{\substack{{X(t_2)=0}\\{X(1)=x}}}\int_{t_2}^{1}L_{\alpha}(\dot{X}(\tau))d\tau\bigg\}
\end{split}
\end{equation*}
then using that $L\geq L_{A}$ for the second term and Jensen's inequality for all terms, we conclude that 
\[\mathcal{D}_0(y,x)\ge \mathcal{D}_{junction}(y,x).\]
Now for $\alpha=\beta,$ we can deduce from the preceding that 
\[\mathcal{D}_0(y,x)\ge \min\left(\mathcal{D}_{junction}(y,x),\inf_{\substack{{X(0)=y}\\{X(1)=x}}}\int_{0}^{1}L_{\alpha}(\dot{X}(\tau))d\tau \right).\]
Then, by Jensen's inequality once again, we can deduce \eqref{J21}. This ends the proof. 
\end{proof}

In view of \eqref{J21}, we see that the study of $\mathcal{D}_0$ can now be reduced to the study of $\mathcal{D}_{junction}.$

\subsection{Study of $\mathcal{D}_{junction}$}
We introduce a simpler notation of $\mathcal{D}_{junction}$ defined in \eqref{J25},
\begin{equation}
\label{J23}
\mathcal{D}_{junction}(y,x)=\inf_{\substack{0\le t_1\le t_2\le 1}} G(t_1,t_2,y,x),
\end{equation}
where  $$G(t_{1},t_{2},y,x)= t_{1}L_{\beta}\bigg(\frac{-y}{t_1}\bigg)+(t_2-t_1)L_A(0)+(1-t_2)L_{\alpha}\bigg(\frac{x}{1-t_2}\bigg).$$
As in \cite{IMZ}, for $(y,x)\in J_{\beta}^{*}\times J_{\alpha}^{*}$  the function $(t_{1},t_{2})\rightarrow G(t_{1},t_{2},y,x)$ is stricly convex on $(0,1)\times (0,1)$. Indeed, for $t_{1},t_{2}\in (0,1)$,
we compute 
$$D^{2}G(t_{1},t_{2},y,x)=\frac{L''_{\beta}\left(\frac{-y}{t_{1}}\right)}{t_{1}}V_{y}^{T}V_{y}+\frac{L''_{\alpha}\left(\frac{x}{1-t_{1}}\right)}{1-t_{2}}V_{x}^{T}V_{x}\geq 0,$$
where $V_{y}=(\frac{-y}{t_{1}},0,1,0)$ and $V_{x}=(0,\frac{x}{1-t_{1}},0,1)$ and in particular, we have
$$\frac{\partial^2}{\partial t_1^2}G(t_{1},t_{2},y,x)= \frac{y^{2}}{t_{1}^{3}}L''_{\beta}\left(\frac{-y}{t_{1}}\right) > 0,$$
and 
$$\frac{\partial^2}{\partial t_2^2}G(t_{1},t_{2},y,x)= \frac{x^{2}}{(1-t_{2})^{3}}L''_{\alpha}\left(\frac{x}{1-t_{1}}\right) > 0.$$
So we deduce that for $(y,x)\in J_{\beta}^{*}\times J_{\alpha}^{*}$, if the function $(t_{1},t_{2})\rightarrow G(t_{1},t_{2},y,x)$ admits a critical point, then it reaches its infimum at this point, else it reaches its infimum at the boundary.  
\begin{lemma}
\label{JLem1}
Let $(y,x)\in J,$ and $D_{junction}(y,x)$ as in \eqref{J25}. We have the following equivalences for the infimum,
\[
\left\lbrace
\begin{array}{ll}
x=0  \Leftrightarrow t_2=1,\\
y=0  \Leftrightarrow t_1=0.
\end{array}
\right.
\]
\end{lemma}

\begin{proof}
It is a direct consequence of the expression \eqref{J25}.
\end{proof}
\begin{defi}[Numbers $\xi_{l}^{+}, \xi_{l}^{-}$]
 We define $\xi_{l}^{-},$ $\xi_{l}^{+}$ thanks to the following function (for $l\in \{1,...N\}$)
\begin{equation}
\label{J33}
K_{l}(x)=L_{l}(x) -xL_{l}^{\prime}(x)-L_A(0).
\end{equation}
We define $(K_{l}^{-})^{-1}$ (resp. $(K_{l}^{+})^{-1}$) as the inverse of the function $K_{l}$ restricted to $(-\infty,0]$ (resp. $[0, +\infty)$), in fact one can write
\[K_{l}^{\prime}(x)=-xL_{l}^{\prime\prime}(x)<0 \text{ on } (0,+\infty)\quad (\text{ resp. }>0 \text{ on } (-\infty,0)).\]
More precisely, we define  $\xi_{l}^{\pm}=(K_{l}^{\pm})^{-1}(0).$
\end{defi}
\begin{lemma}[\textbf{Explicit expression of $\mathcal{D}_{junction}(y,x)$}]
\label{JLem2}
It exists a unique function $\tau:J\times J \rightarrow (0,1)$ of class $\mathcal{C}^{1}$ such that
for $(y,x)\in J_{\beta}\times J_{\alpha},$ we have
\begin{equation}
\label{J26}
\mathcal{D}_{junction}(y,x)=
\left\lbrace
\begin{array}{ll}
\tau(y,x) L_{\beta}\left(\frac{-y}{\tau(y,x)}\right)+(1-\tau(y,x))L_{\alpha}\left(\frac{x}{1-\tau(y,x)}\right)& \text{ if } (y,x)\in (J_{\beta}^{\star}\times J_{\alpha}^{\star})\setminus \Delta_{\beta\alpha},\\
-yL_{\beta}^{\prime}(\xi_{\beta}^{-})+xL_{\alpha}^{\prime}(\xi_{\alpha}^{+})+L_A(0) & \text{ if } (y,x)\in \Delta_{\beta\alpha},\\
L_{\alpha}(x)  & \text{ if } y=0 \text{ and } x>\xi_{\alpha}^{+},\\
L_{\beta}(-y) & \text { if } x=0 \text{ and } y>-\xi_{\beta}^{-},
 \end{array}
\right.
\end{equation}
where \[\Delta_{\beta\alpha}=\bigg\{(y,x)\in J_{\beta}\times J_{\alpha}, \quad \frac{x}{\xi_{\alpha}^{+}}-\frac{y}{\xi_{\beta}^{-}}\le 1 \bigg\}.\]
\end{lemma}

\begin{figure}
\begin{center}
  \includegraphics[width=6cm]{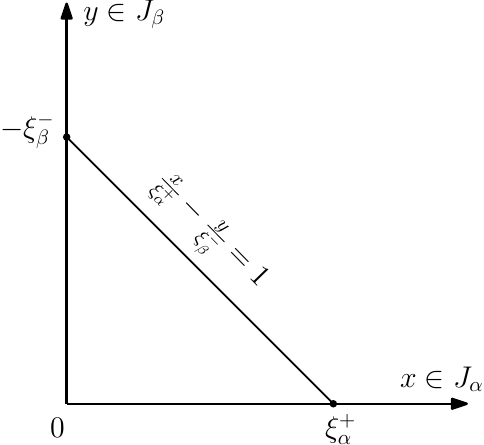}
  \caption{Illustration of the several subsets for $\mathcal{D}_{junction}$ for $\alpha\neq \beta$.}
  \label{alphaneqbeta}
  \end{center}
  \end{figure}
  
  \begin{figure}
\begin{center}
  \includegraphics[width=6cm]{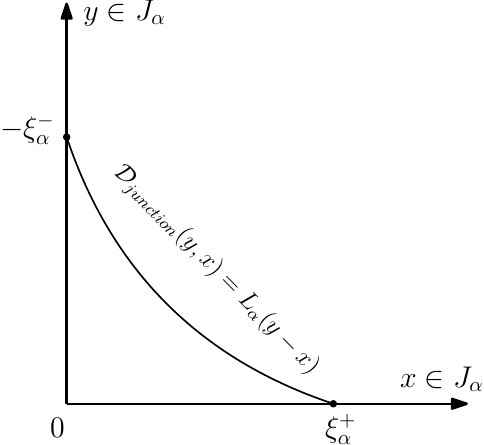}
  \caption{Illustration of the several subsets for $\mathcal{D}_{0}$ for $\alpha= \beta$.}
  \label{alpha=beta}
  \end{center}
  \end{figure}

We have a different expression of $\mathcal{D}_{junction}$ on each subset  of the previous Lemma (see Figure \ref{alphaneqbeta}).  
\begin{proof}
Writing the optimal conditions of $G$ associated with the infimum in \eqref{J23}, we have
\begin{equation}
\label{J24}
\left\lbrace
\begin{array}{ll}
\frac{y}{t_1}L_{\beta}^{\prime}\left(\frac{-y}{t_1}\right)-L_A(0)+L_{\beta}\left(\frac{-y}{t_1}\right)=0,\\ \\
-\frac{x}{1-t_2}L_{\alpha}^{\prime}\left(\frac{x}{1-t_2}\right)-L_A(0)+L_{\alpha}\left(\frac{x}{1-t_2}\right)=0,
\end{array}
\right.
\end{equation}
where $t_1$ and $t_2$ are the quantities realizing the minimum.
Hence from \eqref{J24}, we deduce
\[K_{\beta}\left(-\frac{y}{t_1}\right)=0=K_{\alpha}\left(\frac{x}{1-t_2}\right).\]
But $K_{\beta}$ is a bijection on $(-\infty,0),$ and so is $K_{\alpha}$ on $(0,+\infty).$ Therefore, setting $(K_{\beta}^{-})^{-1}(0):=\xi_{\beta}^{-},$ and 
$(K_{\alpha}^{+})^{-1}(0):=\xi_{\alpha}^{+},$ we deduce for $(y,x)\in \Delta_{\beta \alpha}\backslash \{xy=0\},$ 
\begin{equation*}
\begin{split}
\mathcal{D}_{junction}(y,x)&=\frac{-y}{\xi_{\beta}^-}L_{\beta}(\xi_{\beta}^{-})+\frac{x}{\xi_{\alpha}^{+}}L_{\alpha}(\xi_{\alpha}^{+})+\left(1-\frac{x}{\xi_{\alpha}^+}+\frac{y}{\xi_{\beta}^{-}}\right)L_A(0)\\
&= -yL_{\beta} ^{\prime}(\xi_{\beta} ^{-})+xL_{\alpha}^{\prime}(\xi_{\alpha}^{+})+L_A(0).
\end{split}
\end{equation*}
Now, for $x=0$ and $y<-\xi_{\beta}^{-}$, using the first condition of \eqref{J24}, we deduce that
$$\mathcal{D}_{junction}(y,0)=-yL_{\beta} ^{\prime}(\xi_{\beta} ^{-})+L_A(0).$$
For $x=0$ and $y\geq  -\xi_{\beta}^{-},$ we deduce from Lemma \ref{JLem1}, that $t_2=1$.
Using the first optiml condition in \eqref{J24}, we have $K_{\beta}\left(\frac{-y}{t_{1}} \right)=0$ so $t_{1}=\frac{-y}{\xi_{\beta}^{-}}\geq 1$. We deduce that the optimal condition must be satisfied at the boundary of the set $\{0\leq t_{1}\leq 1\}$. Here using \eqref{J25}, we have $t_{1}=1$, so
 \[D_{junction}(y,0)=L_{\beta}(-y).\]
Similarly, for $y=0$ and $x< \xi_{\alpha}^{+},$
$$\mathcal{D}_{junction}(y,x) = xL_{\alpha}^{\prime}(\xi_{\alpha}^{+})+L_A(0).$$
For $y=0$ and $x\geq \xi_{\alpha}^{+},$ we deduce that 
\[\mathcal{D}_{junction}(0,x)=L_{\alpha}(x).
\]
In all other cases, that is to say for $(y,x)\in (J_{\beta}^{\star}\times J_{\alpha}^{\star})\setminus \Delta_{\beta \alpha},$ the infimum of $G$ is attained at the boundary of $\{0\leq t_{1} \leq t_{2}\leq 1 \}$, here for some $t_{1}=t_{2}=\tau\in(0,1).$ Hence we have
\[
\mathcal{D}_{junction}(y,x)=\inf_{\substack{0< \tau< 1}}\bigg\{\tau L_{\beta}\left(\frac{-y}{\tau}\right)+(1-\tau)L_{\alpha}\left(\frac{x}{1-\tau}\right)\bigg\}
\]
Once again, writing the optimal conditions for $G(\tau,\tau,y,x),$ we deduce that
\begin{equation}
\label{J34}
K_{\beta}\left(\frac{-y}{\tau}\right)=K_{\alpha}\left(\frac{x}{1-\tau}.\right).
\end{equation}
We define \[\tilde{G}(\tau,y,x)=K_{\beta}\left(\frac{-y}{\tau}\right)-K_{\alpha}\left(\frac{x}{1-\tau}\right).\]
Deriving \[\frac{\partial \tilde{G}}{\partial \tau}=K_{\beta}^{\prime}\left(\frac{-y}{\tau}\right)\frac{y}{\tau^2}-K_{\alpha}^{\prime}\left(\frac{x}{1-\tau}\right)\frac{x}{(1-\tau)^2}>0 \quad \mbox{ for } (y,x)\in(J_{\beta}^{\star}\times J_{\alpha}^{\star})\setminus \Delta_{\beta \alpha},\]
by implicit function theorem, we deduce that there exists a unique $\tilde{\tau}\in C^1(0,1)$ satisfying $\tilde{G}(\tilde{\tau},y,x)=0.$ The proof is thus complete.
\end{proof}
\begin{lemma}[\textbf{Continuity of $\mathcal{D}_{junction}$}]
\label{Jcontinuity}
The function $\mathcal{D}_{junction}$ is continuous in $J^2.$
\end{lemma}
\begin{proof}
From \eqref{J26}, we already know that $\mathcal{D}_{junction}\in C((J_{\beta}^{\star}\times J_{\alpha}^{
\star})\backslash\Delta_{\beta\alpha})\cup C(\Delta_{\beta\alpha}\cup\{ x=0\}\cup \{ y=0\}).$ Therefore in order to prove that $\mathcal{D}_{junction}\in C(J_{\beta}\times J_{\alpha}),$ it is sufficient to prove that for any given sequence $(y^k,x^k)\in (J_{\beta}^{\star}\times J_{\alpha}^{\star})\setminus \Delta_{\beta\alpha}$ such that $(y^k,x^k)\rightarrow (y,x),$ where $(y,x)\in\bar{\Delta}:=\{\frac{x}{\xi_{\alpha}^{+}}-\frac{y}{\xi_{\beta}^{-}} = 1\}\cup\{ x\geq \xi_{\alpha}^{+}\}\cup \{ y\geq -\xi_{\beta}^{-}\},$ we have  
\[
\mathcal{D}_{junction}(y^k,x^k)\rightarrow\mathcal{D}_{junction}(y,x).
\]
Since the sequence $\{\tau(y^{k},x^{k})\}$ is bounded, we can deduce that there exists a sub-sequence such that $\tau(y^k,x^k)\rightarrow \tau^0.$ We distinguish the following cases.
\subparagraph{Case 1: $\tau^0\in(0,1).$}   By continuity of $K_{l},$ we have \begin{equation}
\label{J446633}
K_{\alpha}\left(\frac{x}{1-\tau^0}\right)=K_{\beta}\left(\frac{-y}{\tau^0}\right).
\end{equation} 

If $x=0,$ we have as $K_{\alpha}(0)>0$ and $(K_{\beta}^{-})^{-1}$ is increasing
$$\frac{y}{\tau^0}=-(K_{\beta}^{-})^{-1}(K_{\alpha})(0)< -(K_{\beta}^{-})^{-1} (0)=-\xi_{\beta}^{-},$$ 
hence deduce that $(y,0)\notin \bar{\Delta} ,$ so this case is not possible.

Similarly, if $y=0,$ we have 
\[
\frac{x}{1-\tau^0}=(K_{\alpha}^{+})^{-1}(K_{\beta})(0)< (K_{\alpha}^{+})^{-1} (0) =\xi_{\alpha}^{+},
\]
hence deduce that $(0,x)\notin \bar{\Delta} ,$ so this case is not possible.

Now if $(y,x)\in (J_{\beta}^{\star}\times J_{\alpha}^{\star})\cap \bar{\Delta},$ then $\frac{x}{\xi_{\alpha}^{+}}-\frac{y}{\xi_{\beta}^{-}} = 1$ and passing to the limit, we have \eqref{J446633}. We know that $K_{\alpha}(\xi_{\alpha}^{+})=K_{\beta}(\xi_{\beta}^{-})=0$, so if we set $\bar{\tau}=-\frac{y}{\xi_{\beta}^{-}}=1-\frac{x}{\xi_{\alpha}^{+}}$ so $1-\bar{\tau}=\frac{x}{\xi_{\alpha}^{+}}$, we have 
\begin{equation}
\label{Jborddelta}
K_{\beta}\left(\frac{-y}{\bar{\tau}}\right)=0=K_{\alpha}\left( \frac{x}{1-\bar{\tau}} \right).
\end{equation}
By uniqueness of $\tau$ satisfying \eqref{J34}, we deduce that $\tau^{0}=\bar{\tau}$.
So we have \[\mathcal{D}_{junction}(y^{k},x^{k})\rightarrow -yL_{\beta}^{\prime}(\xi_{\beta}^{-})+xL_{\alpha}^{\prime}(\xi_{\alpha}^{+})+L_A(0)=\mathcal{D}_{junction}(y,x).\]

\subparagraph{Case 2: $\tau^0=0.$} In this case using Lemma \ref{JLem1}, $y^k\rightarrow y=0,$ so $x\geq \xi_{\alpha}^{+}$ and with \eqref{J34} we deduce that 
\begin{equation}
\label{Jinv1}
\frac{-y^k}{\tau(y^k,x^k)}=(K_{\beta}^{-})^{-1}\left(K_{\alpha}\left(\frac{x^{k}}{1-\tau(y^{k},x^{k})}\right)\right)\rightarrow (K_{\beta}^{-})^{-1}\left(K_{\alpha}\left(x\right)\right).
\end{equation}
Therefore $\mathcal{D}_{junction}(y^{k},x^{k})\rightarrow L_{\alpha}(x)=\mathcal{D}_{junction}(0,x).$

\subparagraph{Case 3: $\tau^0=1.$} In this case, $x^k\rightarrow x=0.$ Arguing as in Case 2, we deduce that $y\ge \xi_{\beta}^{-},$ and
\begin{equation}
\label{Jinv2}
\frac{x^{k}}{1-\tau(y^{k},x^{k})}=(K_{\alpha}^{+})^{-1}\left(K_{\beta}\left(\frac{-y^k}{\tau(y^k,x^k)}\right)\right)\rightarrow  (K_{\alpha}^{+})^{-1}\left(K_{\beta}\left(-y\right)\right).
\end{equation}
Therefore, $\mathcal{D}_{junction}(y^{k},x^{k}) \rightarrow L_{\beta}(-y)=\mathcal{D}_{junction}(y,x).$

The proof is thus complete.
\end{proof}
\begin{lemma}
\label{JLem5}
The function $\mathcal{D}_{junction}$ is $C^{1}$ in $J^2$ and
for $(y,x)\in J_{\beta}\times J_{\alpha},$ we have
\begin{equation}
\label{J27}
\partial_{x}\mathcal{D}_{junction}(y,x)=
\left\lbrace
\begin{array}{ll}
L_{\alpha}^{\prime}(\frac{x}{1-\tau})& \text{ if } (y,x)\in (J_{\beta}^{\star}\times J_{\alpha}^{\star})\setminus \Delta_{\beta\alpha},\\
L_{\alpha}^{\prime}(\xi_{\alpha}^{+})& \text{ if } (y,x)\in \Delta_{\beta\alpha},\\
L_{\alpha}^{\prime}(x)  & \text{ if } y=0 \text{ and } x>\xi_{\alpha}^{+},\\
L_{\alpha}^{\prime}\circ (K_{\alpha}^{+})^{-1}\circ K_{\beta}(-y) & \text { if } x=0 \text{ and } y>-\xi_{\beta}^{-},
 \end{array}
\right.
\end{equation}
and 
\begin{equation}
\label{J28}
\partial_{y}\mathcal{D}_{junction}(y,x)=
\left\lbrace
\begin{array}{ll}
-L_{\beta}^{\prime}(\frac{-y}{\tau})& \text{ if } (y,x)\in (J_{\beta}^{\star}\times J_{\alpha}^{\star})\setminus \Delta_{\beta\alpha},\\
-L_{\beta}^{\prime}(\xi_{\beta}^{-})  &
 \text{ if } (y,x)\in \Delta_{\beta\alpha},
 \\
-L_{\beta}^{\prime}\circ(K_{\beta}^{-})^{-1}\circ K_{\alpha}(x)& \text{ if } y=0 \text{ and } x>\xi_{\alpha}^{+},\\
-L_{\beta}^{\prime}(-y) & \text { if } x=0 \text{ and } y>-\xi_{\beta}^{-}.
 \end{array}
\right.
\end{equation}
\end{lemma}
\begin{proof}
We compute the partial derivatives in domains where the function is naturally of class $C^{1}$ using that the function $\tau$ is  continuously differentiable in $(0,1)^{2}$ and using \eqref{J34}. We prove the continuity of the partial derivatives using the same proof as Lemma \ref{Jcontinuity}.
\end{proof}

\subsection{Compatibility condition}
In this subsection, we prove a compatibility result, which will be used in deriving error estimates.
Let us introduce the following shorthand notation
$$H(x,p)=\left\{\begin{array}{ll}
H_{\alpha}(p) & \mbox{ if } x\in J_{\alpha}^{*} \\
F_{A}(p) & \mbox{ if } x=0.
\end{array}\right.$$

\begin{remark}
\label{JRem1}
In $J_{\alpha}\times J_{\alpha}$, we give a description of $\{\mathcal{D}_{junction}(y,x)=L_{\alpha}(x-y)\}\cap \Delta_{\alpha\alpha}$ using \cite{IMZ}, see Figure \ref{alpha=beta}. We have
\begin{equation*}
\left\lbrace
\begin{array}{ll}
\mathcal{D}_{junction}(0,\xi_{\alpha}^{+})=\xi_{\alpha}^{+}L_{\alpha}^{\prime}(\xi_{\alpha}^{+})+L_{A}(0)=L_{\alpha}(\xi_{\alpha}^{+}),\\
\mathcal{D}_{junction}(-\xi_{\alpha}^{-},0)=\xi_{\alpha}^{-}L_{\alpha}^{\prime}(\xi_{\alpha}^{-})+L_{A}(0)=L_{\alpha}(\xi_{\alpha}^{-}).
\end{array}
\right.
\end{equation*}
This means that the functions $\mathcal{D}_{junction}$ and $(y,x)\rightarrow L_{\alpha}(x-y)$ coincide at the same points $X_{\alpha}=(0,\xi_{\alpha}^{+})$ and $Y_{\alpha}=(-\xi_{\alpha}^{-},0).$ Therefore we have
\[L_{\alpha}(x-y)<\mathcal{D}_{junction}(y,x)\quad \text{ on the open line segment } ]X_{\alpha},Y_{\alpha}[\]
because $\mathcal{D}_{junction}$ is linear and $L_{\alpha}$ is strictly convex as a function of $y-x.$

The function $(y,x)\mapsto L_{\alpha}(x-y)-\mathcal{D}_{junction}(y,x)$ being convex because $\mathcal{D}_{junction}(y,x)$ is linear, we can consider the convex set
\[K^{\alpha}=\{(y,x)\in J_{\alpha}\times J_{\alpha}, \quad L_{\alpha}(x-y)\le \mathcal{D}_{junction}(y,x)\}.\]
Then the set 
\[\Gamma^{\alpha}=\{(y,x)\in \Delta_{\alpha\alpha},\quad \mathcal{D}_{junction}(y,x)=L_{\alpha}(x-y)\}\]
is contained in the boundary of the convex set $K^{\alpha}.$ More precisely, we have
\[\Gamma^{\alpha}=((\partial K^{\alpha})\cap \Delta_{\alpha\alpha})\subset J_{\alpha}\times J_{\alpha}\]
which shows that $\Gamma^{\alpha}$ is a curve which contains the points $X_{\alpha}$ and $Y_{\alpha}$.
\end{remark}

\begin{theorem}
\label{JthCC}
Assume the Hamiltonians are convex, with Legendre Fenchel transform satisfying (B0). Then for all $(x,y)\in J\times J \backslash \bigcup\limits_{\alpha\in \{1,\dots, N \}} \Gamma^{\alpha},$ (i.e., everywhere except on the curves where $\mathcal{D}_{0}$ is not $C^{1}$), we have
\[H(y,-\partial_{y}\mathcal{D}_{0}(y,x))=H(x,\partial_{x}\mathcal{D}_{0}(y,x)).\]
\end{theorem}

\begin{proof}[Proof of Theorem \ref{JthCC}]
First, notice that in the interior of $K^{\alpha}$ (i.e., in the regions where $\mathcal{D}_{0}(y,x)=L_{\alpha}(x-y)$), we have the result as $$H(y,-\partial_{y}\mathcal{D}_{0}(y,x))=H_{\alpha}(L'_{\alpha}(x-y))=H(x,\partial_{x}\mathcal{D}_{0}(y,x)).$$

Now we prove the result in the regions where $\mathcal{D}_{0}=\mathcal{D}_{junction}$. 
 We distinguish different regions of $J_{\beta}\times J_{\alpha},$ defined in the expressions of $\partial_{x}\mathcal{D}_{junction}$ and $\partial_{y}\mathcal{D}_{junction}$ in \eqref{J27}-\eqref{J28}. Let us first point out that we have the following assertion
\begin{equation}
\label{J32}
H_{\alpha}(p)+L_{\alpha}(q)=pq \Leftrightarrow q\in \partial H_{\alpha}(p),
\end{equation}
where $\partial H_{\alpha}(p)$ is the convex subdifferential of $H_{\alpha}(p).$

We distinguish several cases.
\subparagraph{\textbf{Case 1 $(y,x)\in (J_{\beta}^{\star}\times J_{\alpha}^{\star})\setminus \Delta_{\beta\alpha}.$}} From \eqref{J32}, on one hand,  and from \eqref{J28} we have 
 \[H_{\beta}\left(L_{\beta}^{\prime}\left(\frac{-y}{\tau}\right)\right)=\frac{-y}{\tau}L_{\beta}^{\prime}\left(\frac{-y}{\tau}\right)-L_{\beta}\left(\frac{-y}{\tau}\right).\]
From \eqref{J33}, we have then  $H_{\beta}\left(L_{\beta}^{\prime}\left(\frac{-y}{\tau}\right)\right)=-K_{\beta}\left(\frac{-y}{\tau}\right)-L_{A}(0).$ 

On the other hand, and from \eqref{J27}
\[
H_{\alpha}\left(L_{\alpha}^{\prime}\left(\frac{x}{1-\tau}\right)\right)=\frac{x}{1-\tau}L_{\alpha}^{\prime}\left(\frac{x}{1-\tau}\right)-L_{\alpha}\left(\frac{x}{1-\tau}\right), 
\]
similarly, from \eqref{J33}, we deduce that $H_{\alpha}\left(L_{\alpha}^{\prime}\left(\frac{x}{1-\tau}\right)\right)=-K_{\alpha}\left(\frac{x}{1-\tau}\right)-L_{A}(0).$
Hence, from \eqref{J34}, the compatibility condition.

\subparagraph{\textbf{Case 2 $(y,x)\in (J_{\beta}^{\star}\times J_{\alpha}^{\star})\cap \Delta_{\beta\alpha}.$}} We argue as in Case 1, one can deduce that
\begin{equation*}
\begin{split}
H_{\beta}(L_{\beta}^{\prime}(\xi_{\beta}^{-}))&=-K_{\beta}(\xi_{\beta}^{-})-L_{A}(0)=A\\
H_{\alpha}(L_{\alpha}^{\prime}(\xi_{\alpha}^{+}))&=-K_{\alpha}(\xi_{\alpha}^{+})-L_{A}(0)=A
\end{split}
\end{equation*}
From the definition of $\xi_{\alpha}^{+}$ and $\xi_{\beta}^{-},$ one can deduce the compatibility condition.

\begin{remark}
We deduce that the functions $\pi_{\alpha}^{+}, \pi_{\beta}^{-}$ defined in \cite{IM} satisfy 
$$\pi_{\alpha}^{+}(A)=L_{\alpha}^{\prime}(\xi_{\alpha}^{+})\quad \mbox{ and } \quad  \pi_{\beta}^{-}(A)=L'_{\beta}(\xi_{\beta}^{-}).$$
\end{remark}

\paragraph{\textbf{Case 3 $y=0$ and $x> \xi_{\alpha}^{+}.$}} Let us check the following equality
\[
\max\left(A,\max_{\substack{\beta =1,\dots,N}} H_{\beta}^{-}\left(L_{\beta}^{\prime}\left(\left(K_{\beta}^{-}\right)^{-1}\circ K_{\alpha}(x)\right)\right)\right)=H_{\alpha}\left(L_{\alpha}^{\prime}(x)\right).
\]
On one hand, from the definition of $K_{\beta}^{-},$ we deduce that \[H_{\beta}^{-}\left(L_{\beta}^{\prime}\left(\left(K_{\beta}^{-}\right)^{-1}\circ K_{\alpha}(x)\right)\right)=H_{\beta}\left(L_{\beta}^{\prime}\left(\left(K_{\beta}^{-}\right)^{-1}\circ K_{\alpha}(x)\right)\right),\]
and arguing as previously, we deduce that 
$$H_{\beta}\left(L_{\beta}^{\prime}\left(\left(K_{\beta}^{-}\right)^{-1}\circ K_{\alpha}(x)\right)\right)=-K_{\beta}\left(\left(K_{\beta}^{-}\right)^{-1}\circ K_{\alpha}(x)\right) -L_{A}(0)=-K_{\alpha}(x)-L_{A}(0).$$

On the other hand from \eqref{J32}, we have  $H_{\alpha}(L_{\alpha}^{\prime}(x))=-K_{\alpha}(x)-L_{A}(0).$

And for $x>\xi_{\alpha}^{+},$ we have $H_{\alpha}(L_{\alpha}^{\prime}(x))> H_{\alpha}(L_{\alpha}^{\prime}(\xi_{\alpha}^{+}))=H_{\alpha}(\pi_{\alpha}^{+}(A))=A.$ So one can deduce the compatibility condition.

\paragraph{\textbf{Case 4 $x=0$ and $y> -\xi_{\beta}^{-}.$}} Let us check the following equality
\[
\max\left(A,\max\limits_{\begin{array}{c}
\alpha =1,\dots,N\\
 \alpha\neq \beta
\end{array}} H_{\alpha}^{-}\left(L_{\alpha}^{\prime}\left(\left( K_{\alpha}^{+}\right)^{-1}\circ K_{\beta}(-y)\right)\right), H_{\beta}^{-}\left(L_{\beta}^{\prime}(-y)\right) \right)=H_{\beta}\left(L_{\beta}^{\prime}(-y)\right).
\]
Similarly, as in the previous case, one can deduce that 
$$\max\limits_{\begin{array}{c}
\alpha =1,\dots,N\\
 \alpha\neq \beta
\end{array}} H_{\alpha}^{-}\left(L_{\alpha}^{\prime}\left(\left(K_{\alpha}^{+}\right)^{-1} \circ K_{\beta}(-y)\right)\right)=A_{0}\leq A.$$
And for $y>\-\xi_{\beta}^{-},$ we have $H_{\beta}^{-}(L_{\beta}^{\prime}(-y))> H_{\beta}^{-}(\pi_{\beta}^{-}(A))=A.$

\paragraph{ \textbf{Case 5 $y=0$ and $0<x\le \xi_{\alpha}^{+}.$}} Let us check the following equality
\[\max(A,\max_{\substack{\beta =1,\dots,N}} H_{\beta}^{-}(L_{\beta}^{\prime}(\xi_{\beta}^{-})))=H_{\alpha}(L_{\alpha}^{\prime}(\xi_{\alpha}^{+})).\]
On one hand,  from \eqref{J32} $H_{\alpha}(L_{\alpha}^{\prime}(\xi_{\alpha}^{+}))=-K_{\alpha}(\xi_{\alpha}^{+})-L_{A}(0)=-L_{A}(0)=A.$

On the other hand, $\max\limits_{\beta =1,\dots,N} H_{\beta}^{-}(L_{\beta}^{\prime}(\xi_{\beta}^{-}))= \max\limits_{\beta =1,\dots,N} H_{\beta}^{-}(L_{\beta}^{\prime}(\xi_{\beta}^{-}))= \max\limits_{\beta =1,\dots,N} H_{\beta}^{-}(\pi_{\beta}^{-}(A))=A.$

\paragraph{\textbf{Case 6 $x=0$ and $0<y\le -\xi_{\beta}^{-}.$}} Let us check the following equality
\[\max(A,\max_{\substack{\alpha =1,\dots,N}} H_{\alpha}^{-}(L_{\alpha}^{\prime}(\xi_{\alpha}^{+}))=H_{\beta}(L_{\beta}^{\prime}(\xi_{\beta}^{-})).\]
Similarly, as in Case 5, one can deduce the compatibility condition.

\paragraph{Case 7 $x=0$ and $y=0.$} Let us check the following equality
\[\max(A,\max_{\substack{\beta =1,\dots,N}} H_{\beta}^{-}(L_{\beta}^{\prime}(\xi_{\beta}^{-}))=\max(A,\max_{\substack{\alpha =1,\dots,N}} H_{\alpha}^{-}(L_{\alpha}^{\prime}(\xi_{\alpha}^{+})).\]
In fact, it follows directly from Case 5 and Case 6.

The proof is thus complete.
\end{proof}

\subsection{$C^{1,1}$ estimates for the reduced minimal action}
\label{sec:c11}

In this section, we study the Lipschitz regularity of the gradient of
the reduced minimal action $\mathcal{D}_{0}$. It turns out that
its gradient is indeed Lipschitz if the flux limiter $A$ is not equal
to $A_0$, the minimal flux limiter. Such a technical result will be
used when deriving error estimates. It is also of independent
interest.

\begin{proposition}[\textbf{$C^{1,1}$ estimates for the reduced minimal action}]\label{prop:c11}
  Let $\rho >0$ and assume that the Hamiltonians satisfy \eqref{Jconvex1}
  and \eqref{minH}.  The function $\mathcal{D}_{0}$ associated with the flux limiter $A_0+\rho$
  can be chosen
  $C^{1,1}(J_{K}^2)$ for any $K>0$ where
  $J_{K}^2=\{(x,y)\in J^2: d(0,x)\le K \mbox{ and } d(0,y)\le K \}$. Moreover, there exists
  $C_K$ and $C'_K$ such that
\begin{equation}
\label{estim1D0}
\norm{{ \partial_{xx}\mathcal{D}_{junction}}}_{L^\infty(J^2_K)}\le \frac{C_K}{\min(1,\rho)};
\end{equation}
 and 
 \begin{equation}
      \label{estim2G}
      \norm{H'_{\alpha}(\partial_{x}\mathcal{D}_{junction})\partial_{xx}\mathcal{D}_{junction}}_{L^\infty(J^2_K)}\le \frac{C'_K}{\min(1,\sqrt{\rho})}.
\end{equation}
the constants $C_K$ and $C'_K$ depends only on $K$ and \eqref{Jconvex1}.

Moreover, in the case where for all $\alpha \in \{1, ..., N \}$, $\min H_{\alpha}=A_{0}$, we have 
\begin{equation}
\label{estim2D0}
\norm{{ \partial_{xx}\mathcal{D}_{junction}}}_{L^\infty(J^2_K)}\le C_K.
\end{equation}
\end{proposition}
\begin{proof}[\textbf{Proof}]
In the following $A$ denotes $A_0+\rho$. 
Using \eqref{J27}, we see that $\partial_{xx}\mathcal{D}_{junction}=0$ on $\Delta_{\beta\alpha}$ for all $(\beta,\alpha)\in \{1,\dots, N\}^{2}$ and $\partial_{xx}\mathcal{D}_{junction}(y,x)=L_{\alpha}^{''}(x)$ on $\{0\}\times \{x\in J_{\alpha} \mbox{ } | \mbox{ } x>\xi_{\alpha}^{+}\}$. So
 it is sufficient to prove \eqref{estim1D0} and \eqref{estim2G} on $T:=J_{\beta}^{*}\times J_{\alpha}^{*} \backslash \Delta_{\beta\alpha}$ for all $(\beta,\alpha)\in \{1,\dots,N\}^{2}$. 
By \eqref{J27}, we deduce that on $T$, 
$$\partial_{xx}\mathcal{D}_{junction}(y,x)=\left(\frac{1}{1-\tau(y,x)}+\frac{x}{(1-\tau(y,x))^2}\frac{\partial\tau}{\partial x}(y,x)\right)L_{\alpha}^{\prime\prime}\left(\frac{x}{1-\tau(y,x)}\right).$$
Let us compute also $\frac{\partial\tau}{\partial x}$ using \eqref{J34},
$$ \frac{\partial\tau}{\partial x}(y,x)  =  \frac{\frac{1}{1-\tau(y,x)} K_{\alpha}^{\prime}\left(\frac{x}{1-\tau(y,x)}\right)}{\frac{y}{\tau(y,x)^{2}}K_{\beta}^{\prime}\left(\frac{-y}{\tau(y,x)}\right)-\frac{x}{(1-\tau(y,x))^{2}}K_{\alpha}^{\prime}\left(\frac{x}{1-\tau(y,x)}\right)}.$$
And as $K_{\beta}^{\prime}\left(\frac{-y}{\tau}\right)=\frac{y}{\tau}L_{\beta}^{\prime\prime}\left(\frac{-y}{\tau}\right) \geq 0$ and $K_{\alpha}^{\prime}\left(\frac{x}{1-\tau}\right)=\frac{-x}{1-\tau}L_{\alpha}^{\prime\prime}\left(\frac{x}{1-\tau}\right) \leq 0$ we deduce that
\begin{equation}
\label{Jderivetau}
\frac{\partial\tau}{\partial x}(y,x) =  \frac{\frac{-x}{(1-\tau(y,x))^{2}}L_{\alpha}^{\prime\prime}\left(\frac{x}{1-\tau(y,x)}\right)}{\frac{y^2}{\tau(y,x)^{3}}L_{\beta}^{\prime\prime}\left(\frac{-y}{\tau(y,x)}\right)+\frac{x^2}{(1-\tau(y,x))^{3}}L_{\alpha}^{\prime\prime}\left(\frac{x}{1-\tau(y,x)}\right)}.
\end{equation}
So we have on $T$
\begin{equation}
\label{Jderiveesec1}
\partial_{xx}\mathcal{D}_{junction}(y,x) =  \frac{\frac{y^2}{(1-\tau(y,x))\tau(y,x)^{3}}L_{\alpha}^{\prime\prime}\left(\frac{x}{1-\tau(y,x)}\right)L_{\beta}^{\prime\prime}\left(\frac{-y}{\tau(y,x)}\right)}{\frac{y^2}{\tau(y,x)^{3}}L_{\beta}^{\prime\prime}\left(\frac{-y}{\tau(y,x)}\right)+\frac{x^2}{(1-\tau(y,x))^{3}}L_{\alpha}^{\prime\prime}\left(\frac{x}{1-\tau(y,x)}\right)} \geq 0.
\end{equation}
 As the denominator is a sum of two positive functions, $\partial_{xx}\mathcal{D}_{junction}$ from above by the same numerator over only one term of the denominator. We deduce in these two cases that, 
\begin{equation}
\label{caspart}
 \partial_{xx}\mathcal{D}_{junction}(y,x) \leq \left\{\begin{array}{ll}
 2L_{\alpha}^{\prime\prime}\left(\frac{x}{1-\tau(y,x)}\right) & \mbox{ if } \tau(y,x)\leq \frac{1}{2}\\
 \frac{8y^{2}}{\left(\frac{x}{1-\tau(y,x)}\right)^{2}} L_{\beta}^{\prime\prime}\left(\frac{-y}{\tau(y,x)}\right) & \mbox{ if } \tau(y,x)\geq \frac{1}{2}.
 \end{array}\right.
 \end{equation} 
 Moreover, we have on $T$, $$H_{\alpha}^{\prime}\left(\partial_{x} \mathcal{D}_{junction}(y,x) \right)=H_{\alpha}^{\prime}\left(L_{\alpha}^{\prime}\left(\frac{x}{1-\tau (y,x)}\right) \right)=\frac{x}{1-\tau (y,x)},$$ and 
 $$\frac{x}{1-\tau (y,x)}\partial_{xx}\mathcal{D}_{junction}(y,x) \leq \left\{\begin{array}{ll}
 4 x^{2}L_{\alpha}^{\prime\prime}\left(\frac{x}{1-\tau(y,x)}\right) & \mbox{ if } \tau(y,x)\leq \frac{1}{2}\\
 \frac{8y^{2}}{\frac{x}{1-\tau(y,x)}} L_{\beta}^{\prime\prime}\left(\frac{-y}{\tau(y,x)}\right) & \mbox{ if } \tau(y,x)\geq \frac{1}{2},
 \end{array}\right.$$

In the case $\tau(y,x)\leq \frac{1}{2}$, as $0\leq \frac{x}{1-\tau(y,x)}\leq 2x$, we get the inequality \eqref{estim1D0} and \eqref{estim2G}. 
Let us prove the following lower bound for $(y,x)\in T$, 
\begin{equation}
\label{Jlowbd}
\frac{x}{1-\tau(y,x)}\geq \xi_{\alpha}^{+},
\end{equation}
which helps us for the second case.
For $y \in J_{\beta}$, we see that $x\rightarrow \frac{x}{1-\tau (y,x)}$ has a non-negative derivative using \eqref{Jderivetau}, so it is a non-decreasing function. Therefore to prove \eqref{Jlowbd}, it is sufficient to show it on $\partial T$. Let $(y,x)$ be in $\partial T$. We distinguish three cases.

In the case where $y=0$, necessarily $x\geq \xi_{\alpha}^{+}$ and as $\tau(y,x)\in [0,1]$, we deduce \eqref{Jlowbd}.

In the case where $y\in ]0,-\xi_{\beta}^{-}[$, we have $(y,x) \in \bigg\{(y,x)\in J_{\beta}\times J_{\alpha}, \quad \frac{x}{\xi_{\alpha}^{+}}-\frac{y}{\xi_{\beta}^{-}}= 1 \bigg\}$. So by \eqref{Jborddelta} we deduce that     
$\frac{x}{1-\tau (y,x)}=\xi_{\alpha}^{+}$.

In the case where $y\geq -\xi_{\beta}^{-}$, we have $x=0$. It is enough to prove that 
\begin{equation}
\label{Jliminf}
\liminf\limits_{x'\rightarrow 0} \frac{x'}{1-\tau (y,x')}\geq \xi_{\alpha}^{+}.
\end{equation}
 We have for $(y,x') \in T$,  
$$ K_{\alpha}\left(\frac{x'}{1-\tau (y,x')}\right)= K_{\beta}\left(\frac{-y}{\tau (y,x')} \right)\leq K_{\beta}\left(\frac{-\xi_{\beta}^{-}}{\tau (y,x')} \right),$$
as $K_{\beta}$ is non-decreasing on $]-\infty,0]$. We deduce that
$$\frac{x'}{1-\tau (y,x')}\geq (K_{\alpha}^{+})^{-1}\circ K_{\beta}\left(\frac{-\xi_{\beta}^{-}}{\tau (y,x')} \right),$$
as $(K_{\alpha}^{+})^{-1}$ is non-increasing. 
As $\lim\limits_{x'\rightarrow 0}\tau(y,x')=1$, taking the limit inferior in the preceding inequality gives \eqref{Jliminf}. So we deduce \eqref{Jlowbd} and 
$$\partial_{xx}\mathcal{D}_{junction}(y,x)\leq  \frac{8y^{2}}{(\xi_{\alpha}^{+})^{2}} L_{\beta}^{\prime\prime}\left(\frac{-y}{\tau(y,x)}\right) \quad \mbox{ if } \tau(y,x)\geq \frac{1}{2},$$
$$ \frac{x}{1-\tau(y,x)}\partial_{xx}\mathcal{D}_{junction}(y,x)\leq  \frac{8y^{2}}{\xi_{\alpha}^{+}} L_{\beta}^{\prime\prime}\left(\frac{-y}{\tau(y,x)}\right) \quad \mbox{ if } \tau(y,x)\geq \frac{1}{2}.$$
If $\xi_{\alpha}^{+}>1$, we deduce \eqref{estim1D0}. 
If $\xi_{\alpha}^{+}\leq1$, let us prove that it exists a constant $C>0$ only depending on \eqref{Jconvex1} such that 
\begin{equation}
\label{Jminorxi}
(\xi_{\alpha}^{+})^{2}\geq C\rho. 
\end{equation}

As $A=A_{0}+\rho$ we have 
$$ K_{\alpha}(\xi)=L_{\alpha}(\xi)-\xi L_{\alpha}^{\prime}(\xi)+A_{0}+\rho,$$
and $$K_{\alpha}^{\prime}(\xi)=-\xi L_{\alpha}^{\prime\prime}(\xi).$$
The function $L_{\alpha}^{\prime\prime}$ is bounded on $[0,1]$, it exists $M>0$ such that
$$\gamma \leq L_{\alpha}^{\prime\prime} \leq M.$$
So we have $K_{\alpha}^{\prime}(\xi)\geq -M\xi$.
We integrate from $0$ to $\xi \geq 0$ and get 
\begin{equation}
\label{boundka}
K_{\alpha}(\xi)-K_{\alpha}(0)\geq -M\frac{\xi^{2}}{2}.
\end{equation}
Taking $\xi=\xi_{\alpha}^{+}$, as $K_{\alpha}(\xi_{\alpha}^{+})=0$ and as $L_{\alpha}(0)+A_{0}\geq 0$, we deduce that  
$$(\xi_{\alpha}^{+})^{2}\geq \frac{2}{M}(L_{\alpha}(0)+A_{0}+\rho)\geq \frac{2}{M}\rho .$$ 
So we get \eqref{Jminorxi} and we deduce \eqref{estim1D0} and \eqref{estim2G}.
 
 In the case where for all $\alpha \in \{1, ..., N \}$, $\min H_{\alpha}=A_{0}$, we only have to consider the case $\tau(y,x)\geq \frac{1}{2}$ in \eqref{caspart} since the case $\tau(y,x)\leq \frac{1}{2}$ gives already the bound \eqref{caspart}. In order to get a bound for the term $\frac{8y^{2}}{\left(\frac{x}{1-\tau(y,x)}\right)^{2}}=\frac{8y^{2}}{\left((K_{\alpha}^{+})^{-1}\circ K_{\beta}\left(-\frac{y}{\tau (y,x)}\right)\right)^{2}}$, let us prove that for all $\xi\in [-2K,2K]$, we have 
 \begin{equation}
 \label{boundcasp}
 \frac{\xi^{2}}{\left((K_{\alpha}^{+})^{-1}\circ K_{\beta}(-\xi)\right)^{2}} \leq C_{2K},
 \end{equation}
 where $C_{2K}>0$ is a constant which depends on $K$.
 Let $M_{2K}$ be such that on $[-2K,2K]$ we have for all $\alpha \in \{1, ..., N \}$,
 $$\gamma \leq L_{\alpha}^{\prime\prime} \leq M_{2K}.$$
Replacing $\xi$ by $(K_{\alpha}^{+})^{-1}(\xi)$ in \eqref{boundka}, we deduce that 
$$M_{2K} \frac{\left((K_{\alpha}^{+})^{-1}(\xi)\right)^{2}}{2}\geq -\xi+K_{\alpha}(0).$$
So we have 
$$ M_{2K} \frac{\left((K_{\alpha}^{+})^{-1}\circ K_{\beta}(-\xi)\right)^{2}}{2}\geq -K_{\beta}(-\xi)+K_{\alpha}(0). $$
As for \eqref{boundka}, we have the following inequality 
$$K_{\beta}(0)-K_{\beta}(-\xi)\geq \gamma\frac{\xi^{2}}{2}.$$
So as $K_{\alpha}(0)=K_{\beta}(0)=\rho$ we deduce that 
$$ M_{2K} \frac{\left((K_{\alpha}^{+})^{-1}\circ K_{\beta}(-\xi)\right)^{2}}{2}\geq \gamma\frac{\xi^{2}}{2}+K_{\alpha}(0)-K_{\beta}(0)\geq \gamma\frac{\xi^{2}}{2}. $$
That gives \eqref{boundcasp} and we deduce \eqref{estim2D0}.
\end{proof}

\section{Error estimates}
\label{sec:ee}

{\subsection{Proof of the error estimates}}

To prove Theorem \ref{thm:ee}, we will need the following result whose
classical proof is given in Appendix for the reader's convenience.
\begin{lemma}[\textbf{A priori control}]\label{Lem3}
  Let $T>0$ and let $u^h$ be a sub-solution of the numerical scheme
  \eqref{scheme}-\eqref{eq:cid} and $u$ a super-solution of \eqref{eq:main}-\eqref{eq:ic} 
  satisfying for some $C_T>0$,
  \[u(t,x)\ge -C_{T}(1+d(0,x))\qquad \text{ for } \quad t\in (0,T).\]
  Then there exists a constant $C=C(T)>0$ such that for all $(t,x) \in
  \mathcal{G}_h$, $t \le T$, and $(s,y)\in[0,T)\times J$, we have
\begin{equation}
u^h(t,x)\le u(s,y)+C(1+d(x,y)).
\end{equation}
\end{lemma}
We also need the following result \cite[Lemma 4.4]{IM} where the proof is given in \cite{IM}. 

\begin{lemma}[From non-convex to convex Hamiltonians]
\label{Jlmbeta}
 Let $K\in (0,+\infty).$
Given Hamiltonians $H_{\alpha}:[-K,K]\rightarrow \mathbb{R}$ satisfying \eqref{H}, there exists a function $\beta:\mathbb{R}\rightarrow \mathbb{R}$ such that the functions $\beta\circ H_{\alpha}$ satisfy \eqref{Jconvex1} for $\alpha=1,...,N$. Moreover, we can choose $\beta$ such that $\beta\in C^{2}(\mathbb{R})$ and $\beta'>1$.   
\end{lemma}

\begin{remark}
In \cite[Lemma 4.4]{IM}, the functions $\beta\circ H_{\alpha}$ satisfy in fact the following assumptions
\begin{equation}
\label{Jconvex2}
\left\lbrace \begin{array}{l}
H_{\alpha}\in C^{2}(\mathbb{R}) \text{ with } H''_{\alpha}>0 \text{ on } \mathbb{R},\\
H'_{\alpha}<0 \text{ on } (-\infty,0) \text{ and } H'_{\alpha}>0 \text{ on } (0,+\infty),\\
\lim\limits_{|p|\rightarrow +\infty} \frac{H_{\alpha}(p)}{|p|}= +\infty.
\end{array}
\right.
\end{equation}
which implies \eqref{Jconvex1}. Indeed, in the next proof on error estimates, we only need to consider Hamiltonians on a compact set which only depends on $u_{0}$ and the Hamiltonians $H_{\alpha}$, thanks to the fact that the solution is Lipschitz continuous, see Theorem \ref{thm:exis-cont} and \eqref{gradient}. So on $[-K,K]$, the functions $(\beta\circ H_{\alpha})''$ are bounded by some constant $C>0$. We deduce that the functions $L_{\alpha}$ are of class $C^{2}$ and satisfy $L''_{\alpha}\geq \gamma=\frac{1}{C}$.   
Indeed, from the relation $H_{\alpha}(p)+L_{\alpha}(q)=pq$ with $q=H'_{\alpha}(p)$, one can deduce that $L'_{\alpha}(H'_{\alpha}(q))=q,$ so $$L''_{\alpha}(q)=\frac{1}{H''_{\alpha}\circ (H'_{\alpha})^{-1}(q)} \geq \gamma.$$
\end{remark}

We now turn to the proof of the error estimates in the case of flux-limited junction conditions. 
\begin{proof}[\textbf{Proof of Theorem~\ref{thm:ee}}]
We assume that the Hamiltonians $H_{\alpha}$ satisfy \eqref{H}. Let $u$ be the solution of \eqref{eq:continuous} and $u^{h}$ the solution of the corresponding scheme \eqref{scheme} with $F=F_{A}$. 

In order to get \eqref{eq:ee}, we only prove that
\[u^h(t,x)-u(t,x)\le 
\begin{cases}
C_T (\Delta x)^{1/2} & \text{ if } A>A_0, \\
C_T (\Delta x)^{2/5} & \text{ if } A=A_0 
\end{cases}
\qquad \text{ in } [0,T)\times {J} \cap \mathcal{G}_{h}\] 
since the proof of the {other} inequality is very similar.
{ We are going to prove that 
\begin{equation}\label{eq:avant-opti}
  u^{h}(t,x)-u(t,x)\le
\begin{cases}
 \mathcal{O} \left(\frac{\Delta t}{\nu}\right)
  + \mathcal{O}\left(\frac{\Delta x}{\eps}\right) 
 +\mathcal{O}(\eps)+\mathcal{O}(\nu) & \text{ if } A > A_0, \\
 \mathcal{O} \left(\frac{\Delta t}{\nu}\right) + \mathcal{O}\left(\frac{\Delta x}{\varepsilon\sqrt{\rho}}\right)
  + \mathcal{O}\left(\frac{(\Delta x)^2}{(\eps \rho)^2}\right) +
  \mathcal{O}(\rho) +\mathcal{O}(\eps)+\mathcal{O}(\nu) & \text{ if } A = A_0.
\end{cases}
\end{equation}
which yields the desired inequality by minimizing the right hand side
with respect to $\eps$ and $\nu$ in the case $A>A_0$ and with
respect to $\rho, \eps$ and $\nu$ in the case $A=A_0$.}  
 Let $\beta$ be the function defined in Lemma \ref{Jlmbeta} such that the functions $\beta \circ H_{\alpha}$ satisfy \eqref{Jconvex1}. 
 In the following, we consider that the function $\mathcal{D}_{0}$ is associated to the Hamiltonians $\beta \circ H_{\alpha}$ and to the flux limiter $\beta (A)$ which satisfies $\beta (A) >\beta (A_{0})$ in the case $A>A_{0}$.

The remaining of the proof proceeds in several steps. 

\paragraph{Step 1: Penalization procedure.}
Using the expression of $\mathcal{D}_{0}$ in \eqref{J21} and $\mathcal{D}_{junction}$ in \eqref{J26}, we deduce that it exists $C>0$, such that $\forall x \in J$
$$\mathcal{D}_{0}(0,0)=L_{A}(0)=-A\leq \mathcal{D}_{0}(x,x)\leq C.$$
Let $\tilde{\mathcal{D}}_{0}=\mathcal{D}_{0}+A$, we have that $$0 \leq \tilde{\mathcal{D}}_{0}(x,x)\leq C+A.$$


  For $\eta$, $\delta$, $\varepsilon$, $\nu$ positive constants, let us define
\begin{equation}
\label{J14}
M_{\varepsilon,\delta}=\sup_{\substack{{(t,x)\in  \mathcal{G}_h},\\{(s,y)\in [0,T)\times J}}}
\bigg\{u^h(t,x)-u(s,y)-\varepsilon \tilde{\mathcal{D}}_{0}\left(\frac{y}{\varepsilon},\frac{x}{\varepsilon}\right)
-\frac{(t-s)^2}{2\nu}-\frac{\delta}{2}d^2(y,0)-\frac{\eta}{T-s}\bigg\}
\end{equation}
where the test function $\mathcal{D}_0$ is given in \eqref{J21}.
In this step, we assume that $M_{\varepsilon,\delta}>0.$
Thanks to
Lemma~\ref{Lem3} and the superlinearity of $\mathcal{D}_0$ (see
Lemma \ref{JLem7}), we deduce that for $(x,y)$ such that the quantity in the supremum is larger than $\frac{M_{\varepsilon,\delta}}{2}$, we have
\[
0<\frac{M_{\varepsilon,\delta}}{2}\le C(1+d(y,x))-\varepsilon\frac{\gamma}{2}d^2\left(\frac{y}{\varepsilon}, \frac{x}{\varepsilon}\right)
-\frac{(t-s)^2}{2\nu}-\frac{\delta}{2}d^2(y,0)-\frac{\eta}{T-s}
\] 
which implies in particular 
\[
\frac{\gamma}{2\varepsilon}d^2(y,x)\le C(1+d(y,x)),
\] and $$\frac{\delta}{2}d^2(y,0)\leq C(1+d(y,x)).$$

Notice that in the following, we use the notation $\mathcal{D}_{0}$ instead of $\tilde{\mathcal{D}}_{0}$.  Indeed we deal only with partial derivatives of $\mathcal{D}_{0}$ which are equal to partial derivatives of $ \tilde{\mathcal{D}}_{0}$ and differences between two values of $\mathcal{D}_{0}$ at two points which are equal to differences between two values of $\tilde{\mathcal{D}}_{0}$ at these two points. 

We deduce from the two last inequalities that $d(y,x)$ is bounded and $d(y,0)$ is bounded, so the supremum is reached at some
point $(t,x,s,y)$ where $y\in J_{\beta}$ and $x \in J_{\alpha}$.
This estimate together with the fact that
$-\partial_{y}\mathcal{D}_{0} (\frac{y}{\varepsilon},\frac{x}{\varepsilon})- \delta d(y,0)$ lies in the viscosity
subdifferential of $u(t,\cdot)$ at $x$ and the fact that $\delta d(y,0)$ is bounded, implies that there exists $K>0$
only depending on $\|\nabla u\|_\infty$ (see
Theorem \ref{thm:exis-cont}) such that the point $(t,x,s,y)$
realizing the maximum satisfies 
\begin{equation}
\label{J15}
 \left| \partial_{y}\mathcal{D}_{0} \left(\frac{y}{\varepsilon},\frac{x}{\varepsilon}\right) \right|
 \le K.
\end{equation}
If $\alpha = \beta,$ for $\frac{y}{\varepsilon}$ or $\frac{x}{\varepsilon}$ large, then \eqref{J15} implies 
$$\left| L_{\alpha}^{\prime}\left(\frac{y}{\varepsilon}-\frac{x}{\varepsilon}\right) \right|
 \le K.$$ As $L_{\alpha}$ is superlinear, it implies that $d\left(\frac{y}{\varepsilon},\frac{x}{\varepsilon}\right)\leq C$, for $C>0$ which is sufficient for the use in step 2 of the $C^{1,1}$ estimates as $\mathcal{D}_{0}$ only depends on $d\left(\frac{y}{\varepsilon},\frac{x}{\varepsilon}\right)$ for $\frac{y}{\varepsilon}$ or $\frac{x}{\varepsilon}$ large.
 If $\alpha \neq \beta$, assume by contradiction that $\frac{y}{\varepsilon}$ or $\frac{x}{\varepsilon}$ are not bounded when $\varepsilon\rightarrow 0$. Then using \eqref{J28} and \eqref{J34} we get a contradiction with \eqref{J15}. So $\frac{y}{\varepsilon}$ and $\frac{x}{\varepsilon}$ are bounded by a constant which only depends on $\|\nabla u\|_\infty$ and the Hamiltonians $H_{\alpha}$. 

We want to prove that for $\eta>\eta^{\star}$ (to be determined)
the supremum in \eqref{J14} is attained for $t=0$ or
$s=0.$ We assume that $t>0$ and $s>0$
and we prove that $\eta \le \eta^\star$.

\paragraph{Step 2: Viscosity inequalities.}
Since $t>0$ and $s>0$, we can use
Lemma \ref{lem:visc-ineq} and get the following viscosity
inequalities.  

If $x \neq 0$, then
\begin{multline*}
  \frac{t-s}{\nu}-\frac{\Delta t}{2\nu} +\max \bigg\{ H_{\alpha}^{-}
  \left( \frac{\varepsilon}{\Delta x} \left\{ \mathcal{D}_0\left(\frac{y}{\varepsilon}, \frac{x+\Delta
          x}{\varepsilon} \right)
    -\mathcal{D}_0\left(\frac{y}{\varepsilon},\frac{x}{\varepsilon}\right)\right\}\right)
  ,\\
  H_{\alpha}^{+} \bigg( \frac{\varepsilon}{\Delta x} \bigg\{
      \mathcal{D}_0\bigg(\frac{y}{\varepsilon},\frac{x}{\varepsilon}\bigg)
     -\mathcal{D}_0\bigg(\frac{y}{\varepsilon},\frac{x-\Delta
      x}{\varepsilon}\bigg) \bigg\}\bigg)\bigg\} \le 0.
\end{multline*}
If $x =0$, then 
\[
\frac{t-s}{\nu}-\frac{\Delta t}{2\nu}+\max\bigg(A, \max_{\substack{\beta}}
\left\{H_{\beta}^{-}\left(\frac{\varepsilon}{\Delta x}\bigg\{\mathcal{D}_0\left(\frac{y}{\varepsilon},\frac{\Delta x}\eps\right)-\mathcal{D}_0\left(\frac{y}{\varepsilon},0\right)\bigg\}\right)\right\}\bigg)\le 0.
\]
If $y \neq 0$, then 
\[
-\frac{\eta}{(T-s)^2}+\frac{t-s}{\nu} +H_{\alpha}\left(
-\partial_{y}{\mathcal{D}_0}\left(\frac{y}{\varepsilon},\frac{x}{\varepsilon}\right)-\delta d(y,0)\right)\ge 0.
\]
If $y =0$, then 
\[
-\frac{\eta}{(T-s)^2}+\frac{t-s}{\nu}+ F_{A}\left(-\partial_{y}\mathcal{D}_0\left( 0,\frac{x}{\varepsilon}\right)\right)
\ge 0.
\]

 We now distinguish the case $A>A_0$ and $A=A_0$.

\paragraph{Case $A>A_0$.} Thanks to the $C^{1,1}$ regularity of the function $\mathcal{D}_0$, see
  Proposition~\ref{prop:c11}, and the fact that the functions $H_{\alpha}^{\pm}, H_{\alpha}$ are locally Lipschitz
  we obtain, for $x\in J_{\alpha}$ and $y\in J_{\beta}$ with $\alpha\neq \beta$ (i.e. for $\mathcal{D}_{0}=\mathcal{D}_{junction}$),

\begin{eqnarray}
\label{J2a}\text{ if } x \neq 0, && \frac{t-s}{\nu}-\frac{\Delta
  t}{2\nu}+H_{\alpha}\left(\partial_x \mathcal{D}_0\left(\frac{y}{\varepsilon},\frac{x}{\varepsilon}\right)\right)+O\left(\frac{\Delta
  x}{\varepsilon}\right) \le 0 \\
\label{J2b}
\text{ if } x = 0, &&\frac{t-s}{\nu}-\frac{\Delta t}{2\nu}+F_{A}\left(
\partial_x\mathcal{D}_0\left(\frac{y}{\varepsilon},0\right)\right)
+O\left(\frac{\Delta x}{\varepsilon }\right)\le 0 \\
\label{J1a}
\text{ if } y \neq 0,& & 
\frac{t-s}{\nu} +H_{\beta}\left(-\partial_y \mathcal{D}_0\left(\frac{y}{\varepsilon},\frac{x}{\varepsilon}\right)\right)
+O(\sqrt{\delta})\ge \frac{\eta}{2 T^2}  \\
\label{J1b}
\text{ if } y = 0, &&\frac{t-s}{\nu} +F_{A} \left(-\partial _y \mathcal{D}_0\left(0,\frac{x}{\varepsilon}\right)\right) \ge \frac{\eta}{2 T^2}.
\end{eqnarray}

Now for $(y,x)\in J_{\alpha}\times J_{\alpha},$ from \eqref{J21} and \eqref{J26}, one can deduce that $\mathcal{D}_0$ is in fact $C^2$ far away from the curve $\Gamma^{\alpha}$ defined in Remark \ref{JRem1}, hence the viscosity inequalities \eqref{J2a}-\eqref{J1b} remain true. 

Now we treat the case where $(\frac{y}{\varepsilon},\frac{x}{\varepsilon})$ is near the curve $\Gamma^{\alpha},$ but not on it. 

First if $(\frac{y}{\varepsilon},\frac{x}{\varepsilon})$ is such that $(\frac{y}{\varepsilon},\frac{x}{\varepsilon})\in K^{\alpha}\backslash \Gamma^{\alpha}$ and $(\frac{y}{\varepsilon},\frac{x-\Delta x}{\varepsilon})\notin K^{\alpha},$
we have $$\mathcal{D}_0\left(\frac{y}{\varepsilon},\frac{x-\Delta x}{\varepsilon}\right)\leq L_{\alpha}\left(\frac{x-\Delta x -y}{\varepsilon}\right).$$ So as $H_{\alpha}^{+}$ is non-decreasing, we deduce that 
$$  H_{\alpha}^{+} \bigg( \frac{\varepsilon}{\Delta x} \bigg\{
     L_{\alpha}\bigg(\frac{x-y}{\varepsilon}\bigg)
     -L_{\alpha}\bigg(\frac{x-\Delta x-y
      }{\varepsilon}\bigg) \bigg\}\bigg) \leq  H_{\alpha}^{+} \bigg( \frac{\varepsilon}{\Delta x} \bigg\{
      \mathcal{D}_0\bigg(\frac{y}{\varepsilon},\frac{x}{\varepsilon}\bigg)
     -\mathcal{D}_0\bigg(\frac{y}{\varepsilon},\frac{x-\Delta
      x}{\varepsilon}\bigg) \bigg\}\bigg).$$ 
      Hence the viscosity inequalities \eqref{J2a}-\eqref{J1b} remain true.
 If $(\frac{y}{\varepsilon},\frac{x}{\varepsilon})$ is such that  $(\frac{y}{\varepsilon},\frac{x}{\varepsilon})\notin K^{\alpha} $ and $(\frac{y}{\varepsilon},\frac{x+\Delta x}{\varepsilon})\in K^{\alpha}\backslash \Gamma^{\alpha}$, we have 
$$\mathcal{D}_0\left(\frac{y}{\varepsilon},\frac{x+\Delta x}{\varepsilon}\right)\leq \mathcal{D}_{junction}\left(\frac{y}{\varepsilon},\frac{x+\Delta x}{\varepsilon}\right).$$
So as $H_{\alpha}^{-}$ is non-increasing, we deduce that 
$$  H_{\alpha}^{-} \bigg( \frac{\varepsilon}{\Delta x} \bigg\{\mathcal{D}_{junction}\bigg(\frac{y}{\varepsilon},\frac{x+\Delta
      x}{\varepsilon}\bigg)-
      \mathcal{D}_{junction}\bigg(\frac{y}{\varepsilon},\frac{x}{\varepsilon}\bigg)
      \bigg\}\bigg) \leq  H_{\alpha}^{-} \bigg( \frac{\varepsilon}{\Delta x} \bigg\{
      \mathcal{D}_0\bigg(\frac{y}{\varepsilon},\frac{x+\Delta x}{\varepsilon}\bigg)-
      \mathcal{D}_0\bigg(\frac{y}{\varepsilon},\frac{x}{\varepsilon}\bigg)
      \bigg\}\bigg).$$ 
Hence the viscosity inequalities \eqref{J2a}-\eqref{J1b} remain true.

Now for $(\frac{y}{\varepsilon},\frac{x}{\varepsilon})$ on the curve $\Gamma^{\alpha},$ we get the following viscosity inequalities, using Proposition \ref{propimp}.

If $x\neq 0,$ then
\begin{multline*}
  \frac{t-s}{\nu}-\frac{\Delta t}{2\nu} +\max \bigg\{ H_{\alpha}^{-}
  \left( \frac{\varepsilon}{\Delta x} \left\{ L_{\alpha}\left(\frac{x+\Delta
          x-y}{\varepsilon}  \right)
    -L_{\alpha}\left(\frac{x-y}{\varepsilon}\right)\right\}\right)
  ,\\
  H_{\alpha}^{+} \bigg( \frac{\varepsilon}{\Delta x} \bigg\{
      L_{\alpha}^{\prime}(\xi_{\alpha}^{+})\frac{x}{\varepsilon}
     -L_{\alpha}^{\prime}(\xi_{\alpha}^{+})\bigg(\frac{x-\Delta
      x}{\varepsilon}\bigg) \bigg\}\bigg)\bigg\} \le 0.
\end{multline*}

If $x =0$, then 
\[
\frac{t-s}{\nu}-\frac{\Delta t}{2\nu}+\max\bigg(A, \max_{\substack{\alpha}}
\left\{H_{\alpha}^{-}\left(\frac{\varepsilon}{\Delta x}\bigg\{L_{\alpha}\left(\frac{\Delta x-y}\eps
\right)-L_{\alpha}\left(-\frac{y}{\varepsilon}\right)\bigg\}\right)\right\}\bigg)\le 0.
\]

If $y \neq 0$, then 
\[
-\frac{\eta}{(T-s)^2}+\frac{t-s}{\nu} +\max \bigg\{ H_{\alpha}^{-}
  \left(L_{\alpha}^{\prime}\bigg(\frac{x-y}{\varepsilon}\bigg)-\delta d(y,0)\right),
  H_{\alpha}^{+} \bigg( L_{\alpha}^{\prime}(\xi_{\alpha}^{-})-\delta d(y,0)\bigg)\bigg\}\ge 0.
\]

If $y =0$, then 
\[
-\frac{\eta}{(T-s)^2}+\frac{t-s}{\nu}+ \max\bigg(A,\max_{\substack{\alpha}}H_{\alpha}^{-}\bigg(L_{\alpha}^{\prime}\bigg(\frac{x}{\varepsilon}\bigg)\bigg)\bigg)
\ge 0.
\]

We now simplify the above inequalities, 
  \begin{eqnarray}
\label{J22a}
\text{ if } x \neq 0,&& \frac{t-s}{\nu}-\frac{\Delta
  t}{2\nu}+\max\bigg\{H_{\alpha}^{-}\left(L_{\alpha}^{\prime}\left(\frac{x-y}{\varepsilon}\right)\right), H_{\alpha}^{+}(L_{\alpha}^{\prime}(\xi_{\alpha}^{+}))\bigg\}+O\left(\frac{\Delta x}{\varepsilon}\right)\le 0\\
\label{J22b}
\text{ if } x = 0, &&\frac{t-s}{\nu}-\frac{\Delta t}{2\nu}+\max\bigg(A,\max_{\substack{\alpha}}{H_{\alpha}}^{-}\bigg(L_{\alpha}^{\prime}\bigg(-\frac{y}{\varepsilon}\bigg)\bigg) \bigg)+O\bigg(\frac{\Delta x}{\varepsilon}\bigg)\le 0 \\
\label{J11a}
\text{ if } y \neq 0,&& 
\frac{t-s}{\nu} +\max\bigg\{H_{\alpha}^{-}\left(L_{\alpha}^{\prime}\left(\frac{x-y}{\varepsilon}\right)\right),H_{\alpha}^{+}(L_{\alpha}^{\prime}(\xi_{\alpha}^{-}))\bigg\}
+O(\sqrt{\delta})\ge \frac{\eta}{2 T^2}  \\
\label{J11b}
\text{ if } y = 0, &&\frac{t-s}{\nu} +\max\bigg(A,\max_{\substack{\alpha}} H_{\alpha}^{-}\bigg(L_{\alpha}^{\prime}\bigg(\frac{x}{\varepsilon}\bigg)\bigg)\bigg) \ge \frac{\eta}{2 T^2}.
\end{eqnarray}

Combining these viscosity inequalities and using Theorem \ref{JthCC} with the Hamiltonians $\beta\circ H_{\alpha}$, we deduce the same equalities for the Hamiltonians  $H_{\alpha}$ as $\beta$ is a bijection.
We use also the fact that $H_{\alpha}^{+}(L_{\alpha}^{\prime}(\xi_{\alpha}^{+}))=A$ and $H_{\alpha}^{+}(L_{\alpha}^{\prime}(\xi_{\alpha}^{-}))=A_{0}$, we get in all cases
\[ \eta \le \mathcal{O}\left(\frac{\Delta t}{\nu}\right)  + O \left(\frac{\Delta x}{\eps}\right) 
+ O(\sqrt{\delta}) =: \eta^\star.\]

\paragraph{Case $A=A_0$.} 
In this case the function $\mathcal{D}_{junction}$ is not of class $C^{1,1}$, see Proposition \ref{prop:c11}.
So we consider the function $\mathcal{D}_{0}$ associated with $A=A_{0}+\rho$ where $\rho$ is a small parameter. The main difference with the case $A>A_{0}$ is in the case $x\in J_{\alpha}$ and $y\in J_{\beta}$ with $\alpha\neq \beta$.
We only treat the case $x\in J_\alpha \setminus \{0\}$ and $y\in J_\beta$ with $\alpha\neq \beta$ since in the other cases the arguments are the same as in the proof of the case $A>A_{0}$. 
Since $\mathcal{D}_{0}(\frac{y}{\varepsilon},.)$ is nondecreasing and $H_{\alpha}^{-}(p)=A_{0}$ for $p\geq 0$, and $H_{\alpha}^{+}(p)=H_{\alpha}(p)$ for $p\geq 0$, we have
\begin{equation}
  \frac{t-s}{\nu}-\frac{\Delta t}{2\nu} +
  H_{\alpha}\left( \frac{\varepsilon}{\Delta x} \left\{
      \mathcal{D}_0\bigg(\frac{y}{\varepsilon},\frac{x}{\varepsilon}\bigg)
     -\mathcal{D}_0\bigg(\frac{y}{\varepsilon},\frac{x-\Delta
      x}{\varepsilon}\bigg) \right\} \right) \le 0.
\end{equation}
By using the Taylor expansion of the function $\mathcal{D}_{0}(\frac{y}{\varepsilon},.)$ of class $C^{1}$, there exists $\theta_{1}\in [0,1]$ such that 
 $$H_{\alpha}\left( \frac{\varepsilon}{\Delta x} \left\{
      \mathcal{D}_{0}\left(\frac{y}{\varepsilon},\frac{x}{\varepsilon}\right)
     -\mathcal{D}_{0}\left(\frac{y}{\varepsilon},\frac{x-\Delta
      x}{\varepsilon}\right) \right\} \right)=H_{\alpha}\left( \partial_{x}\mathcal{D}_{0}\left(\frac{y}{\varepsilon},\frac{x}{\varepsilon}\right) -\frac{\Delta x}{2\varepsilon}\partial_{xx}\mathcal{D}_{0}\left(\frac{y}{\varepsilon},\frac{x-\theta_{1}\Delta
      x}{\varepsilon}\right) \right).$$
      Using now a Taylor expansion of the function $H_{\alpha}$ of class $C^{2}$, there exists $\theta_{2}\in [0,1]$ such that 
      \begin{multline}
      \label{term1}
      H_{\alpha}\left( \frac{\varepsilon}{\Delta x} \left\{
      \mathcal{D}_{0}\left(\frac{y}{\varepsilon},\frac{x}{\varepsilon}\right)
     -\mathcal{D}_{0}\left(\frac{y}{\varepsilon},\frac{x-\Delta
      x}{\varepsilon}\right) \right\} \right)= \\ H_{\alpha}\left( \partial_{x}\mathcal{D}_{0}\left(\frac{y}{\varepsilon},\frac{x}{\varepsilon}\right)\right) -\frac{\Delta x}{2\varepsilon}\partial_{xx}\mathcal{D}_{0}\left(\frac{y}{\varepsilon},\frac{x-\theta_{1}\Delta
      x}{\varepsilon}\right)H'_{\alpha}\left( \partial_{x}\mathcal{D}_{0}\left(\frac{y}{\varepsilon},\frac{x}{\varepsilon}\right)\right) \\
      +\frac{1}{8}\left(\frac{\Delta x}{\varepsilon}\right)^{2}\partial_{xx}\mathcal{D}_{0}\left(\frac{y}{\varepsilon},\frac{x-\theta_{1}\Delta
      x}{\varepsilon}\right)^{2}H''_{\alpha}\left(\partial_{x}\mathcal{D}_{0}\left(\frac{y}{\varepsilon},\frac{x}{\varepsilon}\right) -\frac{\theta_{2}\Delta x}{2\varepsilon}\partial_{xx}\mathcal{D}_{0}\left(\frac{y}{\varepsilon},\frac{x-\theta_{1}\Delta
      x}{\varepsilon}\right)\right).
      \end{multline}
      
Using Taylor expansion for $\partial_{x}\mathcal{D}_{0}(.,\frac{y}{\varepsilon})$ and $H'_{\alpha}$ of class $C^{1}$  there exists $\theta_{3},\theta_{4}\in [0,1]$ such that 
\begin{equation}
\label{term2}
\begin{array}{lll}
H'_{\alpha}\left( \partial_{x}\mathcal{D}_{0}\left(\frac{y}{\varepsilon},\frac{x}{\varepsilon}\right)\right) & = & H'_{\alpha}\left(\partial_{x}\mathcal{D}_{0}\left(\frac{y}{\varepsilon},\frac{x-\theta_{1}\Delta x}{\varepsilon}\right)+\theta_{1}\frac{\Delta x}{\varepsilon}\partial_{xx}\mathcal{D}_{0}\left(\frac{y}{\varepsilon},\frac{x-\theta_{3}\Delta x}{\varepsilon}\right)\right)\\
 & = & H'_{\alpha}\left(\partial_{x}\mathcal{D}_{0}\left(\frac{y}{\varepsilon},\frac{x-\theta_{1}\Delta x}{\varepsilon}\right)\right)\\
  & & +\theta_{1}\frac{\Delta x} {\varepsilon}\partial_{xx}\mathcal{D}_{0}\left(\frac{y}{\varepsilon},\frac{x-\theta_{3}\Delta x}{\varepsilon}\right)H''_{\alpha}\left(\partial_{x}\mathcal{D}_{0}\left(\frac{y}{\varepsilon},\frac{x-\theta_{1}\Delta x}{\varepsilon}\right)+\theta_{4}\frac{\Delta x}{\varepsilon}\partial_{xx}\mathcal{D}_{0}\left(\frac{y}{\varepsilon},\frac{x-\theta_{3}\Delta x}{\varepsilon}\right)\right).
\end{array} 
\end{equation}

Notice that the terms in $H_{\alpha}^{\prime\prime}$ are bounded since $\frac{x}{\varepsilon}$, $\frac{y}{\varepsilon}$ and $\frac{\Delta x}{\varepsilon \rho}$ are bounded independentely of $\Delta x\leq 1$ as we take $\varepsilon=\rho=\Delta x^{\frac{2}{5}}$. 

So combining \eqref{term1} and \eqref{term2}, thanks to the $C^{1,1}$ regularity of the function $\mathcal{D}_{0}$, see
  Proposition \ref{prop:c11} we deduce that 
  $$H_{\alpha}\left( \frac{\varepsilon}{\Delta x} \left\{
      \mathcal{D}_{0}\left(\frac{y}{\varepsilon},\frac{x}{\varepsilon}\right)
     -\mathcal{D}_{0}\left(\frac{y}{\varepsilon},\frac{x-\Delta
      x}{\varepsilon}\right) \right\} \right)= H_{\alpha}\left( \partial_{x}\mathcal{D}_{0}\left(\frac{y}{\varepsilon},\frac{x}{\varepsilon}\right)\right) + \mathcal{O}\left(\frac{\Delta x}{\varepsilon\sqrt{\rho}}\right) +\mathcal{O}\left(\left(\frac{\Delta x}{\varepsilon\rho}\right)^{2}\right).$$
So combining the viscosity inequality and using the fact that $|F_A - F_{A_0}| \le \rho$ we have
\begin{equation}\label{sigma*2}
 \eta \le \mathcal{O}\left(\frac{\Delta t}{\nu}\right) + \mathcal{O}\left(\frac{\Delta x}{\varepsilon\sqrt{\rho}}\right) +\mathcal{O}\left(\left(\frac{\Delta x}{\varepsilon\rho}\right)^{2}\right)
+ \mathcal{O}(\sqrt{\delta}) + \rho =: \eta^\star.
\end{equation}

\paragraph{Step 3: Estimate of the supremum.}
We proved in the previous step that, if $\eta > \eta^\star$, then either $M_{\varepsilon,\delta}\leq 0$ or
$M_{\eps,\delta}$ is reached either for $t=0$ or $s=0$.

If $t=0$, then using Theorem \ref{thm:exis-cont}, we have
\[M_{\varepsilon,\delta}\le u_0(x)-u_0(y) -\frac{\gamma}{2\varepsilon}d^{2}(y,x)+C_{T}s-\frac{s^2}{2\nu}.\]
Using the fact that $u_0$ is $L_0$-Lipschitz, one can deduce
\[\begin{split}
M_{\varepsilon,\delta}&\le \sup\limits_{r\geq 0} \left(L_0 r-\frac{\gamma}{2\varepsilon}r^{2}\right) + \sup_{r>0} \left(Cr - \frac{r^2}{2\nu}\right)\\
&\le O(\varepsilon)+O(\nu).\end{split}\]

If $s=0$, then we can argue similarly (by using the stability of the numerical scheme Lemma \ref{lem:stab} and get 
\[ M_{\eps,\delta} \le O(\eps) +O(\nu).\]

\paragraph{Step 4: Conclusion.} We proved that for $\eta >
\eta^\star$, $M_{\eps,\delta} \le O(\eps) +
O(\nu)$. This implies that for all $(t,x) \in \mathcal{G}_h$, $t \le T/2$, 
we have
\[u^h(t,x)-u(t,x)\le \varepsilon
\tilde{\mathcal{D}_0}\left(\frac{x}{\varepsilon},\frac{x}{\varepsilon}\right)
+\frac{\delta}{2}d^2(x,0)+\frac{2\eta}{T}+ O(\eps)
+O(\nu)\] Replacing $\eta$ by $2\eta^{\star}$ and
recalling that $\tilde{\mathcal{D}}_0(x,x)\le C+A$ for all $x\in J$, we deduce that for
$(t,x) \in \mathcal{G}_\eps$ and $t\le T/2$ (after letting $\delta \to 0$),
\[  u^h(t,x)-u(t,x)\le O \left(\frac{\Delta t}{\nu}\right)
  + O\left(\frac{\Delta x}{\eps}\right) 
  +O(\varepsilon)+O(\nu).
\]
Using the CFL condition~\eqref{cflr} and optimizing with respect to
$\eps$ and $\nu$ yields the desired result on $[0,\frac{T}{2})$. Doing the whole proof with $u$ the solution of \eqref{eq:continuous} and $u^{h}$ the solution of the corresponding scheme \eqref{scheme} with $F=F_{A}$ on $[0,2T)$ yields the desired result on $[0,T)$.

\end{proof}

\begin{remark}
If for all $\alpha \in \{1, ..., N \}$, $\min H_{\alpha}=A_{0}$, then 
in the case where $A=A_0$, thanks to the $C^{1,1}$ regularity of the function $D_{0}$, see
  Proposition~\ref{prop:c11}, we can conclude as the case $A>A_{0}$ that the error estimate is of order $\Delta x^{\frac{1}{2}}$. 
\end{remark}

\subsection{Numerical simulations}

In this subsection, we give a numerical example which illustrates the convergence rate we obtained in the previous subsection. 
In the case $A>A_{0}$, we get an optimal error estimate of order $\Delta x^{\frac{1}{2}}$. But in the case $A=A_{0}$ we only have examples with an error estimate of order $\Delta x^{\frac{1}{2}}$ when in the proof we have $\Delta x^{\frac{2}{5}}$. So we wonder if the error estimate obtained in the proof is optimal for the case $A=A_{0}$.

Here we consider a junction with two branches $J_{1}=J_{2}=[0,X]$. We have the two following Hamiltonians, 
$$ H_{1}(p)=p^{2},$$
$$ H_{2}(p)=p^{2}-1,$$
and the initial data
$$u_{0}(x)=\left\{\begin{array}{ll}
\sin(0.2x) & \mbox{ if } x\in J_{1},\\
\sin(x) & \mbox{ if } x\in J_{2}.
\end{array}
\right. $$
In the simulation we take $X=0.1$ and we give the error $||u(T,.)-u^{h}(T,.)||_{\infty}$ at time $T=0.01$. 
Here we have $A_{0}=0$.

For $A=0$,  $A=0.1>A_{0}$ and $\Delta t=\frac{\Delta x}{10}$  we get the following result, see Figure \ref{casA0} and \ref{casA} ploted in logarithmic scale and the error values in Table \ref{tablA0} and \ref{tablA}.
\begin{figure}
\begin{center}
\begin{tabular}{|c|c|}
  \hline
  $\Delta x$  & $||u(T,.)-u^{h}(T,.)||_{\infty}$ \\
  \hline
  0.00250 & 1,192$\times 10^{-4}$ \\
   0.00100   &  0,753$ \times 10^{-4}$ \\ 
   0.00075  & 0,644$ \times 10^{-4}$ \\
   0.00050  & 0,503$ \times 10^{-4} $ \\
   0.00025  & 0,329 $ \times 10^{-4}$ \\
  \hline
\end{tabular}
 \caption{Error estimates for $A=A_{0}=0$}
 \label{tablA0}
\end{center}
\end{figure}
\begin{figure}
\begin{center}
\begin{tabular}{|c|c|}
  \hline
  $\Delta x$  & $||u(T,.)-u^{h}(T,.)||_{\infty}$ \\
  \hline
  0.00250 & 1,266$\times 10^{-4}$ \\
   0.00100   &  0,719$ \times 10^{-4}$ \\ 
   0.00075  & 0,616$ \times 10^{-4}$ \\
   0.00050  & 0,511$ \times 10^{-4} $ \\
   0.00025  & 0,350 $ \times 10^{-4}$ \\
  \hline
\end{tabular}
 \caption{Error estimates for $A=0.1>A_{0}$}
 \label{tablA}
\end{center}
\end{figure}
\begin{figure}
\begin{center}
  \includegraphics[width=12.0cm]{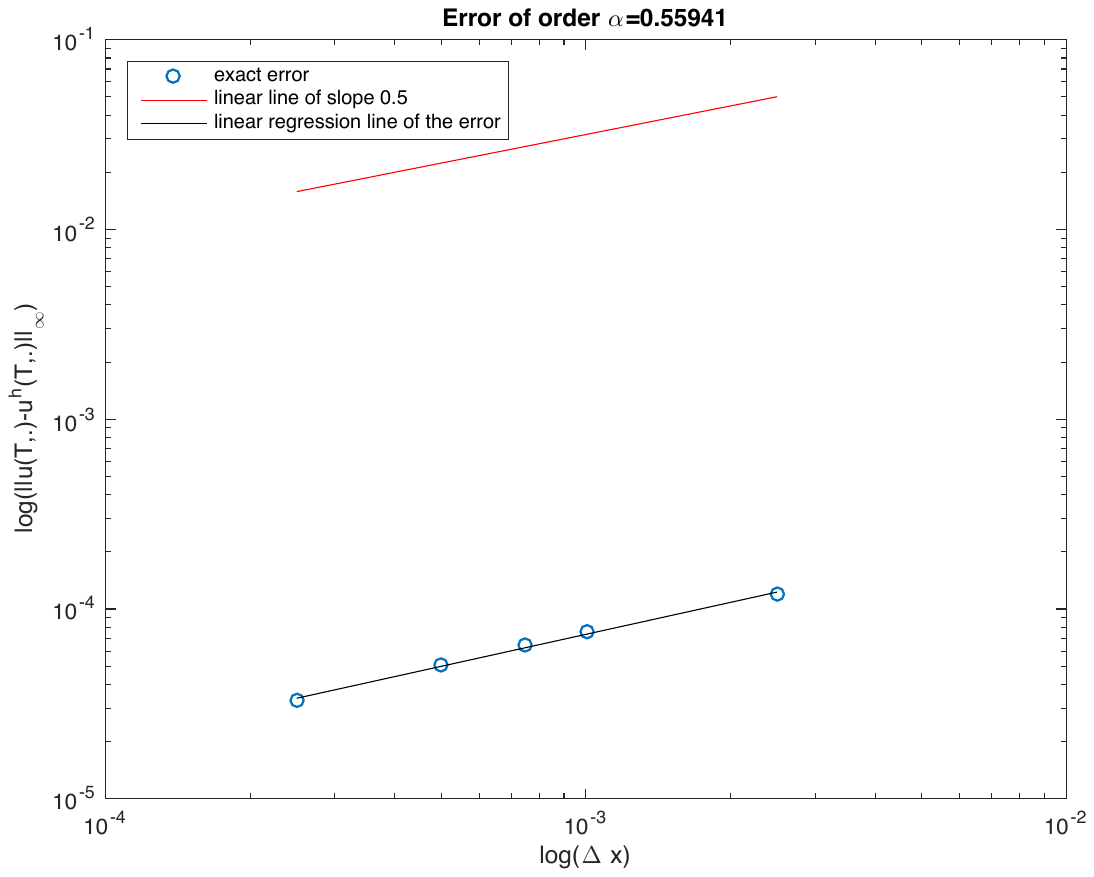}
  \caption{Error estimates for $A=A_{0}=0$}
  \label{casA0}
  \end{center}
  \end{figure} \begin{figure}
\begin{center}
  \includegraphics[width=12.0cm]{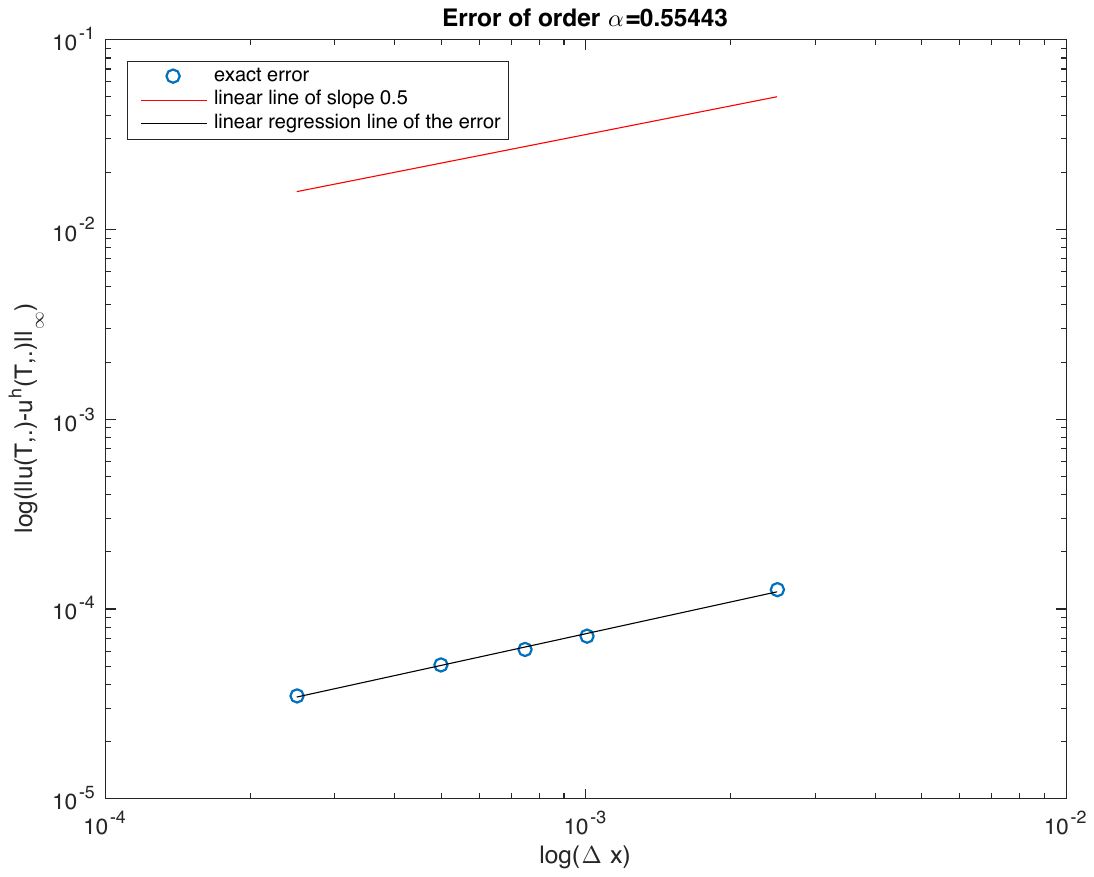}
  \caption{Error estimates for $A=0.1>A_{0}$}
  \label{casA}
  \end{center}
  \end{figure}

\appendix

\section{Proofs of some technical results} 

\subsection{Proof of a priori control}
\label{app:tech}

In order to prove Lemma \ref{Lem3}, we need the following one.
\begin{lemma}[\textbf{A priori control at the same time}]\label{Lem7}
  Assume that $u_0$ is Lipschitz continuous.  Let $T>0$ and let
  $u^h$ be a sub-solution of \eqref{scheme}-\eqref{eq:cid} and $u$
  be a super-solution of \eqref{eq:main}-\eqref{eq:ic}. Then there
  exists a constant $C=C_{T}>0$ such that for all $(t,x) \in
  \mathcal{G}_h$, $t \le T$, $y\in
  J$, we have
\begin{equation}
u^h(t,x)\le u(t,y)+C_{T}(1+d(x,y)).
\end{equation}
\end{lemma}
We first derive Lemma \ref{Lem3} from Lemma \ref{Lem7}.
\begin{proof}[\textbf{Proof of Lemma \ref{Lem3}}]
Let us fix some $h$ and let us consider the sub-solution
$u^{-}$ of \eqref{scheme} and the super-solution $u^{+}$ of of
\eqref{eq:main} defined as :
\[u^{+}(t,x)=u_0(x)+C_{0}t\]
\[u^{-}(n\Delta t,i\Delta x)=u_0(i\Delta x)-C_{0}n\Delta t\] where
\[C_{0}=\max\bigg\{\abs{A},\max_{\substack{\alpha=1,\dots,N}}
\max_{\substack{\abs{p_{\alpha}}\le
    L_{0}}}\abs{H_{\alpha}(p_{\alpha})};\max_{\substack{\abs{p_{\alpha}}\le
    L_{0}}}F(p_1,\dots,p_N)\bigg\}\]
 and $L_0$ denotes the Lispchitz
constant of $u_0$.  We have for all
$(t,x)\in\mathcal{G}_{h}$, with $t\leq T$, $(s,y)\in[0,T)\times J$
\[u^{-}(t,x)-u^{+}(s,y)\le 2C_{0}T+L_{0}d(x,y).\]
We first apply Lemma \ref{Lem7} to control $u^h(t,x)-u^{-}(t,x)$
and then apply Lemma \ref{Lem3} to control $u^{+}(s,y)-u(s,y).$
Finally we get the control on $u^{h}(t,x)-u(s,y)$.
\end{proof}
We can now prove Lemma~\ref{Lem7}. 
\begin{proof}[\textbf{Proof of Lemma \ref{Lem7}}]
We define $\varphi$ in $J^2$ as
\[ \varphi(x,y)=\sqrt{1+d^2(x,y)}.\] 
Since,
\[d^2(x,y)=\left\lbrace\begin{array}{l}
(x-y)^2 \qquad \text{if}\quad (x,y)\in J_{\alpha}\times J_{\alpha}\\
(x+y)^2 \qquad \text{if}\quad (x,y)\in J_{\alpha}\times J_{\beta} \mbox{ with } \alpha\neq \beta
\end{array}\right.\] 
we see that $d^2$ (and consequently $\varphi$) is in $C^{1,1}$ in
$J^2$. Moreover $\varphi$ satisfies
\begin{equation}
\label{fi}
\abs{\varphi_{x}(x,y)},\abs{\varphi_{y}(x,y)}\le 1.
\end{equation}

Recalling that there exists $C>0$ such that 
\[ |u^h(t,x) -u_0(x) | \le C t \quad \text{ and } \quad |u(t,y) - u_0(y)| \le Ct \]
(see Theorem~\ref{thm:exis-cont} and \eqref{eq:estimuo}) and using that $u_0$ is $L_0$-Lipschitz continuous
we deduce that for all $(t,x) \in
  \mathcal{G}_h$, $t \le T$, $y\in
  J$,
$$u^h(t,x)-u(t,y)\leq 2Ct+L_{0}\varphi(x,y),$$
which yields the desired result.
\end{proof}
\subsection{Construction of $\tilde{F}$}\label{AppendixB}
\begin{lemma}
There exists $\tilde{F}$, such that 
\begin{enumerate}
\item $\tilde{F}$ satisfies \eqref{F};
\item $F=\tilde{F}$ in $Q_{0}$;
\item \label{3} For $p \in \rr^N$,
$(-\nabla \cdot \tilde F)(p) \le \sup_{Q_0} (-\nabla \cdot F)$.
\end{enumerate} 
\end{lemma}
\begin{proof}
Let $I_{\alpha}$ denote $[\underline{p}_{\alpha}^{0}; \overline{p}_\alpha]$ so that $Q_{0}=\prod_{\alpha} I_{\alpha}.$ 
We then define 
\[\tilde{F}(p)=F(\mathcal{P}_{1}(p_{1}),\dots,\mathcal{P}_{N}(p_N))
-\sum_{\alpha=1}^N C_{\alpha}(p_{\alpha}-\mathcal{P}_{\alpha}(p_\alpha)),
\] 
where 
\[C_{\alpha}= \min_{\substack{p\in Q_{0}}} 
\left(-\frac{\partial F}{\partial p_{\alpha}}(p_{1},\dots,p_{N})\right), \] and 
\[\mathcal{P}_{\alpha}(r)=
\begin{cases}
\underline{p}_{\alpha}^{0} & \text{ if } r<\underline{p}_{\alpha}^{0},\\
r & \text{ if } r \in I_\alpha, \\
\overline{p}_\alpha & \text{ if } r>\overline{p}_\alpha.
\end{cases}\] Remark that in view of the assumptions made on $F$, we
have $C_\alpha >0$ which will ensure that \eqref{F} holds true.
%
It is now easy to check that \eqref{F} and Item~\ref{3} are satisfied. 
This ends the proof of the Lemma. 
\end{proof} 

\subsection{Relation between the junction and BLN conditions}
\label{app:bln}

Consider the following scalar conservation law posed on $(0,+\infty)$,
\[ \left\{ \begin{array}{ll}
\partial_t v + \partial_x (H(v)) = 0, & t>0, x >0, \\
v (t,0) = v_b (t), & t>0, \\
v(0,x) = v_0 (x), & x >0.
\end{array} \right. \]
The usual BLN condition asserts that the trace $v_\tau$ of the entropy
solution at $x=0$ (if it exists) of the previous scalar conservation
law should satisfy
\[ \forall \kappa \in [\min(v_b,v_\tau),\max (v_b,v_\tau)], \qquad
\sgn (v_\tau - v_b) ( H(v_\tau)-H (\kappa) ) \le 0. \] If $H$ is
quasi-convex, this reduces to
\[ H (v_\tau ) = \max (H^-(v_\tau),H^+(v_b)).\]
This corresponds to a flux limiter $A = H^+ (v_b)$. 

\paragraph{Acknowledgements.} 
The authors acknowledge Cyril Imbert for his expert advice and encouragement
throughout working on this research paper. 
The authors acknowledge the support of
Agence Nationale de la Recherche through the funding of the project
HJnet ANR-12-BS01-0008-01. The second author’s PhD thesis is supported
by UL-CNRS Lebanon.

\bibliographystyle{plain}
\bibliography{ee}

\begin{thebibliography}{10}

\bibitem{ACCT}
Yves Achdou, Fabio Camilli, Alessandra Cutr{\`{\i}}, and Nicoletta Tchou.
\newblock Hamilton-{J}acobi equations constrained on networks.
\newblock {\em NoDEA Nonlinear Differential Equations Appl.}, 20(3):413--445,
  2013.

\bibitem{ad-mi-go}
Adimurthi, Siddhartha Mishra, and G.~D.~Veerappa Gowda.
\newblock Optimal entropy solutions for conservation laws with discontinuous
  flux-functions.
\newblock {\em J. Hyperbolic Differ. Equ.}, 2(4):783--837, 2005.

\bibitem{ags}
Boris Andreianov, Paola Goatin, and Nicolas Seguin.
\newblock Finite volume schemes for locally constrained conservation laws.
\newblock {\em Numer. Math.}, 115(4):609--645, 2010.
\newblock With supplementary material available online.

\bibitem{bar-dol}
Martino Bardi and Italo Capuzzo-Dolcetta.
\newblock {\em Optimal control and viscosity solutions of
  {H}amilton-{J}acobi-{B}ellman equations}.
\newblock Systems \& Control: Foundations \& Applications. Birkh\"auser Boston,
  Inc., Boston, MA, 1997.
\newblock With appendices by Maurizio Falcone and Pierpaolo Soravia.

\bibitem{bln}
C.~Bardos, A.~Y. le~Roux, and J.-C. N{\'e}d{\'e}lec.
\newblock First order quasilinear equations with boundary conditions.
\newblock {\em Comm. Partial Differential Equations}, 4(9):1017--1034, 1979.

\bibitem{BS}
G.~Barles and P.E. Souganidis.
\newblock Convergence of approximation schemes for fully nonlinear second order
  equations.
\newblock {\em Asymptotic Anal.}, 4(3), 1991.

\bibitem{B}
Guy Barles.
\newblock {\em Solutions de viscosit\'e des \'equations de
  {H}amilton-{J}acobi}, volume~17 of {\em Math\'ematiques \& Applications
  (Berlin) [Mathematics \& Applications]}.
\newblock Springer-Verlag, Paris, 1994.

\bibitem{cfs}
Fabio Camilli, Adriano Festa, and Dirk Schieborn.
\newblock An approximation scheme for a {H}amilton-{J}acobi equation defined on
  a network.
\newblock {\em Appl. Numer. Math.}, 73:33--47, 2013.

\bibitem{canseg}
Cl{\'e}ment Canc{\`e}s and Nicolas Seguin.
\newblock Error estimate for {G}odunov approximation of locally constrained
  conservation laws.
\newblock {\em SIAM J. Numer. Anal.}, 50(6):3036--3060, 2012.

\bibitem{C}
I.~Capuzzo~Dolcetta.
\newblock On a discrete approximation of the {H}amilton-{J}acobi equation of
  dynamic programming.
\newblock {\em Appl. Math. Optim.}, 10(4):367--377, 1983.

\bibitem{CDI}
I.~Capuzzo-Dolcetta and H.~Ishii.
\newblock Approximate solutions of the {B}ellman equation of deterministic
  control theory.
\newblock {\em Appl. Math. Optim.}, 11(2):161--181, 1984.

\bibitem{cgs}
Christophe Chalons, Paola Goatin, and Nicolas Seguin.
\newblock General constrained conservation laws. {A}pplication to pedestrian
  flow modeling.
\newblock {\em Netw. Heterog. Media}, 8(2):433--463, 2013.

\bibitem{cg}
Rinaldo~M. Colombo, Paola Goatin, and Massimiliano~D. Rosini.
\newblock Conservation laws with unilateral constraints in traffic modeling.
\newblock In {\em Applied and industrial mathematics in {I}taly {III}},
  volume~82 of {\em Ser. Adv. Math. Appl. Sci.}, pages 244--255. World Sci.
  Publ., Hackensack, NJ, 2010.

\bibitem{CLM}
Guillaume Costes\`eque, Jean-Patrick Lebacque, and R{\'e}gis Monneau.
\newblock A convergent scheme for {H}amilton-{J}acobi equations on a junction:
  application to traffic.
\newblock {\em Numer. Math.}, 129(3):405--447, 2015.

\bibitem{CL}
M.~G. Crandall and P.-L. Lions.
\newblock Two approximations of solutions of {H}amilton-{J}acobi equations.
\newblock {\em Math. Comp.}, 43(167):1--19, 1984.

\bibitem{CIL}
Michael~G. Crandall, Hitoshi Ishii, and Pierre-Louis Lions.
\newblock User's guide to viscosity solutions of second order partial
  differential equations.
\newblock {\em Bull. Amer. Math. Soc. (N.S.)}, 27(1):1--67, 1992.

\bibitem{div}
J.~Droniou, C.~Imbert, and J.~Vovelle.
\newblock An error estimate for the parabolic approximation of multidimensional
  scalar conservation laws with boundary conditions.
\newblock {\em Ann. Inst. H. Poincar\'e Anal. Non Lin\'eaire}, 21(5):689--714,
  2004.

\bibitem{F}
M.~Falcone.
\newblock A numerical approach to the infinite horizon problem of deterministic
  control theory.
\newblock {\em Appl. Math. Optim.}, 15(1):1--13, 1987.

\bibitem{Fal-fer}
Marizio Falcone and Roberto Ferretti.
\newblock Discrete time high-order schemes for viscosity solutions of
  {H}amilton-{J}acobi-{B}ellman equations.
\newblock {\em Numer. Math.}, 67(3):315--344, 1994.

\bibitem{gs}
Marguerite Gisclon and Denis Serre.
\newblock \'{E}tude des conditions aux limites pour un syst\`eme strictement
  hyberbolique via l'approximation parabolique.
\newblock {\em C. R. Acad. Sci. Paris S\'er. I Math.}, 319(4):377--382, 1994.

\bibitem{gzh}
Simone G{\"o}ttlich, Ute Ziegler, and Michael Herty.
\newblock Numerical discretization of {H}amilton-{J}acobi equations on
  networks.
\newblock {\em Netw. Heterog. Media}, 8(3):685--705, 2013.

\bibitem{gues}
Olivier Gu{\`e}s.
\newblock Perturbations visqueuses de probl\`emes mixtes hyperboliques et
  couches limites.
\newblock {\em Ann. Inst. Fourier (Grenoble)}, 45(4):973--1006, 1995.

\bibitem{IM}
Cyril Imbert and R{\'e}gis Monneau.
\newblock {Flux-limited solutions for quasi-convex Hamilton-Jacobi equations on
  networks}.
\newblock {\em Annales scientifiques de l'ENS, to appear}, February 2016.
\newblock 103 pages.

\bibitem{IMZ}
Cyril Imbert, R{\'e}gis Monneau, and Hasnaa Zidani.
\newblock A {H}amilton-{J}acobi approach to junction problems and application
  to traffic flows.
\newblock {\em ESAIM Control Optim. Calc. Var.}, 19(1):129--166, 2013.

\bibitem{collegelions}
P.-L. Lions.
\newblock {Lectures at College de France}.
\newblock 2015-2016.

\bibitem{Lions}
Pierre-Louis Lions.
\newblock {\em Generalized solutions of {H}amilton-{J}acobi equations},
  volume~69 of {\em Research Notes in Mathematics}.
\newblock Pitman (Advanced Publishing Program), Boston, Mass.-London, 1982.

\bibitem{ov}
Mario Ohlberger and Julien Vovelle.
\newblock Error estimate for the approximation of nonlinear conservation laws
  on bounded domains by the finite volume method.
\newblock {\em Math. Comp.}, 75(253):113--150, 2006.

\bibitem{schieborn}
Dirk Schieborn.
\newblock {\em Viscosity Solutions of {H}amilton-{J}acobi Equations of Eikonal
  Type on Ramified Spaces}.
\newblock PhD thesis, Eberhard-Karls-Universit\"at T\"ubingen, 2006.

\bibitem{CS}
Dirk Schieborn and Fabio Camilli.
\newblock Viscosity solutions of {E}ikonal equations on topological networks.
\newblock {\em Calc. Var. Partial Differential Equations}, 46(3-4):671--686,
  2013.

\end{thebibliography}

\end{document}